\documentclass[12pt,a4paper]{amsart}

\usepackage[all]{xy}
\usepackage{amssymb}

\allowdisplaybreaks[2]

\def\<{\langle}
\def\>{\rangle}
\let\ipscriptstyle=\scriptscriptstyle
\def\lipsqueeze{{\mskip -3.0mu}}
\def\ripsqueeze{{\mskip -3.0mu}}
\def\ipcomma{\nobreak\mathrel{,}\nobreak}
\newbox\ipstrutbox
\setbox\ipstrutbox=\hbox{\vrule height8.5pt
width 0pt}
\def\ipstrut{\copy\ipstrutbox}
\def\lip#1<#2,#3>{\mathopen{\relax_{\ipstrut\ipscriptstyle{
#1}}\lipsqueeze
\langle} #2\ipcomma #3 \rangle}
\def\blip#1<#2,#3>{\mathopen{\relax_{\ipstrut
\ipscriptstyle{ #1}}\lipsqueeze\bigl\langle} #2\ipcomma #3 \bigr\rangle}
\def\rip#1<#2,#3>{\langle #2\ipcomma #3
\rangle_{\ripsqueeze\ipstrut\ipscriptstyle{#1}}}
\def\brip#1<#2,#3>{\bigl\langle #2\ipcomma #3
\bigr\rangle_{\ripsqueeze\ipstrut\ipscriptstyle{#1}}}
\def\angsqueeze{\mskip -6mu}
\def\smangsqueeze{\mskip -3.7mu}
\def\trip#1<#2,#3>{\langle\smangsqueeze\langle #2\ipcomma #3
\rangle\smangsqueeze\rangle_{\ripsqueeze\ipstrut\ipscriptstyle{#1}}}
\def\btrip#1<#2,#3>{\bigl\langle\angsqueeze\bigl\langle #2\ipcomma
#3
\bigr\rangle
\angsqueeze\bigr\rangle_{\ripsqueeze\ipstrut\ipscriptstyle{#1}}}
\def\tlip#1<#2,#3>{\mathopen{\relax_{\ipstrut\ipscriptstyle{
#1}}\lipsqueeze \langle\smangsqueeze\langle} #2\ipcomma #3
\rangle\smangsqueeze\rangle}
\def\btlip#1<#2,#3>{\mathopen{\relax_{\ipstrut\ipscriptstyle{
#1}}\lipsqueeze
\bigl\langle\angsqueeze\bigl\langle} #2\ipcomma #3
\bigr\rangle\angsqueeze\bigr\rangle}

\def\ip(#1|#2){(#1\mid #2)}
\def\bip(#1|#2){\bigl(#1 \mid #2\bigr)}
\def\Bip(#1|#2){\Bigl( #1 \bigm| #2 \Bigr)}

\def\notnu{\rho}

\newtheorem{theorem}{Theorem}[section]
\newtheorem{thm}[theorem]{Theorem}
\newtheorem{lemma}[theorem]{Lemma}
\newtheorem{prop}[theorem]{Proposition}
\newtheorem{cor}[theorem]{Corollary}

\theoremstyle{remark}
\newtheorem{example}[theorem]{Example}
\newtheorem{remark}[theorem]{Remark}

\newtheorem*{problem*}{Problem}
\newtheorem*{remark*}{Remark}
\newtheorem*{convention*}{Convention}
\newtheorem*{notation*}{Notation}
\newtheorem*{examples*}{Examples}
\newtheorem*{example*}{Example}
\newtheorem*{warning*}{Warning}

\numberwithin{equation}{section}

\def\DD{{\mathcal D}}
\def\Cc{{\mathcal C}}

\def\K{{\mathcal K}}
\def\L{{\mathcal L}}

\newcommand{\nd}{\textup{nd}}
\newcommand{\Fix}{\operatorname{\mathtt{Fix}}}
\newcommand{\RCP}{\operatorname{\mathtt{RCP}}}
\newcommand{\CP}{\operatorname{\mathtt{CP}}}
\newcommand{\CC}{\operatorname{\mathtt{C*}}}
\newcommand{\CCnd}{\operatorname{\CC_{\!\!\!\!\nd}}}
\newcommand{\Aa}{\operatorname{\mathtt{C*act}}}
\newcommand{\Aca}{\operatorname{\mathtt{C*coact}}}
\newcommand{\Graph}{\operatorname{\mathtt{DG}}}
\newcommand{\AGraph}{\operatorname{\mathtt{DGact}}}
\newcommand{\FAGraph}{\operatorname{\mathtt{DGfreeact}}}
\newcommand{\Qq}{\operatorname{\mathtt{Q}}}
\newcommand{\CK}{\operatorname{\mathtt{CK}}}
\newcommand{\nor}{\textup{n}}
\newcommand\tensor{\otimes}

\newcommand{\rt}{\textup{rt}}
\newcommand{\id}{\textup{id}}
\newcommand{\Ad}{\operatorname{Ad}}

\newcommand{\Aut}{\operatorname{Aut}}

\newcommand{\newspan}{\operatorname{span}}
\newcommand{\clsp}{\overline{\operatorname{span}}}

\def\Mor{\textup{Mor}}

\newcommand{\ib}{im\-prim\-i\-tiv\-ity bi\-mod\-u\-le}

\begin{document}

\title[Morita equivalence for proper actions]{Naturality of Rieffel's Morita equivalence\\ for proper actions}

\author[an Huef]{Astrid an Huef}
\address{School of Mathematics and Statistics\\
The University of New South Wales\\Sydney\\
NSW 2052\\
Australia}
\email{astrid@unsw.edu.au}

\author[Kaliszewski]{S. Kaliszewski}
\address{Department of Mathematical and Statistical Sciences\\Arizona
State University\\Tempe\\ AZ
85287-1804\\USA} \email{kaliszewski@asu.edu}

\author[Raeburn]{Iain Raeburn}
\address{School  of Mathematics and Applied Statistics, University of
Wollongong, NSW 2522, Australia}
\email{raeburn@uow.edu.au}

\author[Williams]{Dana P. Williams}
\address{Department of Mathematics\\Dartmouth College\\
Hanover, NH 03755\\USA}
\email{dana.williams@dartmouth.edu}

\begin{abstract}
Suppose that a locally compact group $G$ acts freely and properly on the right of a locally compact space $T$. Rieffel proved that if $\alpha$ is an action of $G$ on a $C^*$-algebra $A$ and there is an equivariant embedding of $C_0(T)$ in $M(A)$, then the action $\alpha$ of $G$ on $A$ is proper, and the crossed product $A\rtimes_{\alpha,r}G$ is Morita equivalent to a generalised fixed-point algebra $\Fix(A,\alpha)$ in $M(A)^\alpha$. We show that the assignment $(A,\alpha)\mapsto\Fix(A,\alpha)$ extends to a functor $\Fix$ on a category of $C^*$-dynamical systems in which the isomorphisms are Morita equivalences, and that Rieffel's Morita equivalence implements a natural isomorphism between a crossed-product functor and $\Fix$. From this, we deduce naturality of Mansfield imprimitivity for crossed products by coactions, improving results of Echterhoff-Kaliszewski-Quigg-Raeburn and Kaliszewski-Quigg-Raeburn, and naturality of a Morita equivalence for graph algebras due to Kumjian and Pask.
\end{abstract}

\subjclass[2000]{46L55}
\date{October 13, 2008}

\thanks{This research was supported by the Australian Research Council
  and the Edward Shapiro fund at Dartmouth College.}

\maketitle

\section{Introduction}

When $\alpha$ is an action of a compact group $G$ on a unital
$C^*$-algebra $A$, there is a large fixed-point algebra $A^\alpha$
which plays an important role in the analysis of the crossed
product $A\rtimes_\alpha G$. When $G$ is locally compact and $A$ is not unital, there may
be very few fixed points: for example, when
$G$ acts on a commutative $C^*$-algebra $C_0(T)$, fixed points would
be functions which are constant on $G$-orbits in $T$, and such
functions need not vanish at infinity. There has nevertheless
been considerable interest in situations where there is a useful
analogue of $A^\alpha$ in the multiplier algebra $M(A)$ \cite{exel,aHRWproper,
  meyer, sit, rw85, proper, integrable}. Here we are particularly
interested in the \emph{proper actions} introduced by Rieffel in
\cite{proper}; in the motivating example, $A=C_0(T)$ is commutative,
$G$ acts properly on the right of $T$, $\alpha$ is the action $\rt$ of
$G$ by right translation on functions, and the algebra $C_0(T/G)$, which we can
view as a subalgebra of $C_b(T)=M(C_0(T))$, is an excellent substitute
for the missing fixed-point algebra.

An action $\alpha:G\to \Aut A$ is proper in the sense of \cite{proper}
if there is a dense $\alpha$-invariant $*$-subalgebra $A_0$ of $A$ which has properties like
those of the subalgebra $C_c(T)$ of $C_0(T)$. In particular,
$A_0$ carries an $M(A)$-valued inner product, and the completion
$Z(A,\alpha)$ of $A_0$ in this inner product is a
Hilbert module over a \emph{generalised fixed-point algebra}
$A^\alpha\subset M(A)$; Rieffel identified a class of \emph{saturated}
proper actions for which $Z(A,\alpha)$ implements a Morita equivalence
between $A^\alpha$ and the reduced crossed product
$A\rtimes_{\alpha,r}G$ \cite[Corollary~1.7]{proper}. In the motivating
example $(C_0(T),\rt)$, the action is proper in Rieffel's sense if
and only if the underlying action of $G$ on $T$ is proper, and
saturated if and only if $G$ acts freely on $T$ (see
\cite[\S3]{marrae}, for example), and the generalised fixed-point algebra is $C_0(T/G)$.

There are by now many examples of proper actions, and there is often
an obvious choice for the dense subalgebra $A_0$. However, there is in
general no systematic method for finding a suitable $A_0$, and thus
Morita equivalences obtained from the construction of \cite{proper}
ostensibly depend on choices.  More recently, though, Rieffel has
shown that if $G$ acts properly on $T$ and there is a nondegenerate
homomorphism $\phi:C_0(T)\to M(A)$ such that $\phi\circ\rt=\alpha$,
then $(A,\alpha)$ is proper with respect to
\[
A_0:=\newspan\{\phi(f)a\phi(g):a\in A, f,g\in C_c(T)\}
\]
(see \cite[Theorem~5.7]{integrable}). There are several situations where there are canonical choices for the system $(C_0(T),\rt)$ and
the homomorphism $\phi$, and then one can ask questions about the
functoriality and naturality of the constructions in \cite{proper}.

To ask questions about functoriality and naturality, one first needs to decide what categories one is working in. The objects in our underlying category $\CC$ are $C^*$-algebras, and the morphisms from $A$ to $B$ are isomorphism classes of the right-Hilbert $A$\,--\,$B$ bimodules which were invented by Rieffel \cite{rieff} to place the theory of induced representations of groups in a $C^*$-algebraic setting. The category $\CC$ was introduced in \cite{taco} and \cite{enchilada} as a setting for imprimitivity theorems for crossed products by actions and coactions, and was independently discovered in other contexts by Landsman \cite{landsman1,landsman2} and by Schweizer  \cite{schweizer}. It has the attractive feature that the invertible morphisms are those which are based on imprimitivity bimodules (see, for example, \cite[Proposition~2.6]{taco}).

Here we extend $\CC$ to a category $\Aa(G)$ of dynamical systems $(A,\alpha)$, show that the constructions $(A,\alpha)\mapsto A\rtimes_{\alpha,r}G$ and $(A,\alpha)\mapsto A^\alpha$ are the object maps in functors from a subcategory of $\Aa(G)$ to $\CC$, and prove that Rieffel's bimodules $Z(A,\alpha)$ implement a natural isomorphism between these functors. We then illustrate our general theorem with two applications. We first prove that the version of Mansfield's imprimitivity theorem for crossed products by coactions in \cite{hrman} is natural, thereby extending one of the main theorems of \cite{enchilada} from normal subgroups to arbitrary subgroups. We then show that if $G$ acts freely on a directed graph $E$, then Kumjian and Pask's Morita equivalence of $C^*(E)\rtimes G$ and $C^*(E/G)$ from \cite{kpaction} is also natural. In a sequel we plan to discuss further applications to nonabelian duality, and in particular to the representation theory of crossed products by coactions of homogeneous spaces and to the duality of restriction and induction, as in \cite{kqrold} and \cite[\S5.2]{enchilada}.

Our starting point is recent work \cite{kqcat, kqrproper} concerning characterisations of crossed-product $C^*$-algebras due to Landstad \cite{landstad} and Quigg \cite{quigg}. Quigg's theorem identifies crossed products by coactions of $G$ as the systems $(A,\alpha)$ which carry an equivariant embedding $j_G:(C_0(G),\rt)\to (M(A),\alpha)$. In \cite{kqrproper}, it was observed that adding the equivariant
homomorphism $j_G$ to the system $(A,\alpha)$ gives an element of
what is called a \emph{comma category}. In general, if $b$ is an object in a
category ${\mathtt D}$, then in the comma category $b\downarrow {\mathtt D}$ the
objects are pairs $(c,\phi)$ with $\phi\in \Mor(b,c)$, and the
morphisms from $(c,\phi)$ to $(d,\psi)$ are morphisms $\theta:c\to d$
in ${\mathtt D}$ such that $\psi=\theta\circ\phi$. The objects in the main
category $\Aa_{\nd}(G)$ of \cite{kqrproper} are systems $(A,\alpha)$,
the morphisms from $(A,\alpha)$ to $(B,\beta)$ are nondegenerate
homomorphisms $\psi:A\to M(B)$ such that $\psi\circ\alpha=\beta$, and
the data $((A,\alpha),j_G)$ in Quigg's theorem defines an
object in the comma category $(C_0(G),\rt)\downarrow
\Aa_{\nd}(G)$. The theorem of Rieffel we discussed above (\cite[Theorem~5.7]{integrable}) says that if $G$ acts properly on $T$, then systems in
the comma category $(C_0(T),\rt)\downarrow\Aa_{\nd}(G)$ are proper in the sense of
\cite{proper}; if $G$ acts freely on $T$, then they are also
saturated, so that \cite[Corollary~1.7]{proper} gives a Morita
equivalence $Z(A,\alpha)$. Corollary~2.8 of \cite{kqrproper} says that
$A\mapsto A^\alpha$ is the object map in a
functor $\Fix$ on the comma category $(C_0(T),\rt)\downarrow
\Aa_{\nd}(G)$; Theorem~3.2 of \cite{kqrproper} says that the assignment $(A,\alpha)\mapsto Z(A,\alpha)$ gives a
natural isomorphism between $\Fix$ and the reduced crossed-product
functor $\RCP$ --- provided we view $\Fix$ and $\RCP$ as taking values in $\CC$.

Our first task is to find a suitable analogue of the comma category
for the category $\Aa(G)$ (as opposed to the category $\Aa_{\nd}(G)$ used in \cite{kqrproper}). Choosing the right
category is crucial: to prove that a system $(A,\alpha)$ is
proper, we need a morphism from $(C_0(T),\rt)$ to $(A,\alpha)$
which is implemented by a nondegenerate homomorphism rather than a
right-Hilbert bimodule, and this forces some asymmetry in our
hypotheses. We call our choice the \emph{semi-comma category}
associated to $(C_0(T),\rt)$, and denote it $\Aa(G,(C_0(T),\rt))$:
the objects in $\Aa(G,(C_0(T),\rt))$ consist of a system
$(A,\alpha)$ and a nondegenerate equivariant homomorphism
$\phi:C_0(T)\to M(A)$, and the morphisms are the usual morphisms in
the category $\CC$. (In the semi-comma category,
unlike the comma category, we do \emph{not} require that the morphisms
are in any way compatible with the homomorphisms $\phi$.) We discuss the semi-comma category and our reasons for choosing it in \S\ref{sec-semicomma}. A key
technical result (Corollary~\ref{factor}) says that every morphism in the semi-comma category factors as a composition of a Morita equivalence and a
morphism which comes from one in the comma category
$(C_0(T),\rt)\downarrow \Aa_{\nd}(G)$, and to which the results of
\cite{kqrproper} apply.

To extend Rieffel's map $(A,\alpha,\phi)\mapsto A^\alpha$ to a functor $\Fix$ from the semi-comma category $\Aa(G,(C_0(T),\rt))$ to $\CC$, we need to handle morphisms. We already know from \cite[Theorem~2.6]{kqrproper} that every nondegenerate homomorphism restricts to a nondegenerate homomorphism between fixed-point algebras, so our first task is to say how we intend to $\Fix$  imprimitivity bimodules. We do this in \S\ref{sec-fix}, using a linking-algebra argument. We then define $\Fix$ on general morphisms by factoring them, dealing with the two pieces, and reassembling. At this point we can state our two main theorems: Theorem~\ref{thm-fix}, which says that $\Fix$ is indeed a functor, and Theorem~\ref{thm-main}, which says that Rieffel's bimodules implement a natural equivalence between $\Fix$ and a reduced crossed-product functor.

We prove our main theorems in \S\ref{sec-proofs}. The hard bit is proving that $\Fix$ing morphisms respects composition. To see why this is hard, notice that it includes as a special case the assertion that when $X$ and $Y$ are imprimitivity bimodules, $\Fix(X\otimes_B Y)$ is isomorphic to $(\Fix X)\otimes_{\Fix B}(\Fix Y)$ --- indeed, we need to prove this as an important step in the proof of Theorem~\ref{thm-fix}. This part of the paper contains some interesting technical innovations: we make heavy use of linking algebras, and introduce a $3\times 3$ matrix trick which we think may be useful elsewhere (see \S\ref{2ibs}).

Our first application is in \S\ref{sec-appl}. We consider the category $\Aca^{\nor}(G)$ of normal coactions $(B,\delta)$ of $G$ and a closed subgroup $H$ of $G$. We deduce from \cite[\S3.1.2]{enchilada} that there is a crossed-product functor $\CP:\Aca^{\nor}(G)\to \Aa(G,(C_0(G),\rt))$ which takes $(B,\delta)$ to $(B\rtimes_{\delta} G,\hat\delta, j_G)$, and then applying Theorem~\ref{thm-main} with $(T,G)=(G,H)$ shows that Rieffel's bimodules give a natural isomorphism between $(B\rtimes_\delta G)\rtimes_{\hat\delta} H$ and $\Fix_H(B\rtimes_\delta G)$. To deduce naturality of Manfield imprimitivity from this we need to identify $\Fix_H(B\rtimes_\delta G)$ with the crossed product $B\rtimes_{\delta} (G/H)$, which we do in Proposition~\ref{idcpGH}. The final result is Theorem~\ref{newnatforcoact}, which directly improves Theorem~6.2 of \cite{kqrproper} by replacing the category $\CCnd$ with $\CC$, and Theorem~4.3 of \cite{enchilada} by removing the requirement that the subgroup is normal.

In the last section, we give our new application to graph algebras, which is naturality of the Morita equivalence of \cite[Theorem~1.6]{pr} (see Theorem~\ref{prnatural}). As in \S\ref{sec-appl}, we have to work to identify $\Fix$ with the functor appearing in \cite{pr}, which takes an action of $G$ on a graph $E$ to the $C^*$-algebra $C^*(E/G)$ of the quotient graph. In this application we are working with group actions on discrete sets, and the semi-comma category appears only as a technical device in the proof of Theorem~\ref{prnatural}.

\section{The semi-comma category}\label{sec-semicomma}

In our basic category $\CC$ the objects are $C^*$-algebras and the
morphisms from $A$ to $B$ are isomorphism classes $[{}_AX_B]$ of
right-Hilbert $A$\,--\,$B$ bimodules $X$. We always assume
that the left action of $A$ on $X$ is implemented by a nondegenerate
homomorphism $\kappa_A:A\to \L(X)$, and that $X$ is a full Hilbert
module, in the sense that the ideal $\newspan\{\langle
x,y\rangle_B:x,y\in X\}$ is dense in $B$. It is proved in
\cite[\S2]{taco} that $\CC$ is a category with composition defined by
$[{}_BY_C]\,[{}_AX_B]=[{}_A(X\otimes_B Y)_C]$ and the identity
morphism at the object $A$ given by $[{}_AA_A]$; it is shown in
\cite[Proposition~2.6]{taco}, for example, that the isomorphisms in
$\CC$ are the classes whose elements are imprimitivity bimodules. We
observe straight away that this category $\CC$ is not the
same\footnote{Nor is it the same as the category denoted by $\Cc$ in
  \cite{enchilada}, where the right-Hilbert bimodules ${}_AX_B$
  implementing the morphisms are not required to be full as right
  Hilbert modules.}  as the category $\CCnd$ used in
\cite{kqrproper} (and there denoted by $\Cc$), in which the objects are $C^*$-algebras and the morphisms
from $A$ to $B$ are nondegenerate homomorphisms from $A$ to
$M(B)$; every morphism $\psi:A\to B$ in $\CCnd$ gives rise to a
morphism $[\psi]$ in $\CC$ with underlying right-Hilbert bimodule
$B_B$ and the left action given by $\psi$ (the difference
between $\psi$ and $[\psi]$ is explained in
\cite[Proposition~2.3]{taco}).

Throughout we consider a fixed locally compact group $G$, and several
associated categories. It is shown in \cite[Proposition~3.3]{taco}
that there is a category $\Aa(G)$ whose objects are dynamical systems
$(A,\alpha)$ consisting of an action $\alpha$ of $G$ on a
$C^*$-algebra $A$, and whose morphisms from $(A,\alpha)$ to
$(B,\beta)$ are isomorphism classes $[X,u]$ of right-Hilbert $A$\,--\,$B$
bimodules $X$ which carry an \emph{$\alpha$--$\beta$ compatible action} $u$ of $G$ satisfying
\[
u_s(a\cdot x\cdot b)=\alpha_s(a)\cdot u_s(x)\cdot\beta_s(b)\ \text{ and }\ \langle u_s(x),u_s(y)\rangle_B=\beta_s(\langle x,y\rangle_B).
\]
Again,
$\Aa(G)$ is not the same as the category $\Aa_{\nd}(G)$ used in
\cite{kqrproper}, where the morphisms are given by nondegenerate
equivariant homomorphisms.

We now fix a free and proper right action of $G$ on a locally compact
space $T$, and consider the action $\rt:G\to \Aut C_0(T)$ defined by
$\rt_s(f)(t)=f(t\cdot s)$. In the \emph{semi-comma category} $\Aa(G,
(C_0(T),\rt))$, the objects $(A,\alpha,\phi_A)$ are systems
$(A,\alpha)$ in $\Aa(G)$ together with a nondegenerate homomorphism
$\phi_A:C_0(T)\to M(A)$ which is $\rt$\,--\,$\alpha$ equivariant, and the
morphisms from $(A,\alpha,\phi_A)$ to $(B,\beta,\phi_B)$ are just the
usual morphisms $[X,u]$ from $(A,\alpha)$ to $(B,\beta)$ in $\Aa(G)$,
with the same composition defined by balanced tensor product of
right-Hilbert bimodules. It follows immediately from
\cite[Proposition~3.3]{taco} that the semi-comma category is indeed a
category. This may seem an unusual choice of category, and we will say
more at the end of the section about our reasons for choosing it (see Remark~\ref{pleaseexplain}). However, it is easy
to explain why the morphism $\phi_A$ is there: it allows us to deduce
from Theorem~5.7 of \cite{integrable} that the system $(A,\alpha)$ is
proper in the sense of \cite{proper} with respect to the subalgebra
$A_0:=\newspan\phi_A(C_c(T))A\phi_A(C_c(T))$.

A key technical result says that every morphism in the semi-comma
category factors into an isomorphism (that is, a morphism implemented
by an imprimitivity bimodule) and a morphism which comes from a
nondegenerate homomorphism. This is proved in
\cite[Proposition~2.27]{enchilada}, which we can
restate in our terms as follows.

\begin{prop}\label{prop-enchilada}
  Let $[X,u]:(A,\alpha)\to(B,\beta)$ be a morphism in
  $\Aa(G)$. Then there is a unique action $\mu:G\to \K(X_B)$ such that
  $(X,u)$ is a $(\K(X),\mu)$\,--\,$(B,\beta)$ imprimitivity bimodule. The
nondegenerate homomorphism $\kappa_A:A\to \L(X)=M(\K(X))$ given by the left action
  is $\alpha$\,--\,$\mu$ equivariant, and $({}_A(\K(X)\otimes_B X)_B,
  \mu\otimes u)$ is isomorphic to $({}_AX_B,u)$, so that $[X,u]$
  is the composition $[{}_{\K(X)}X_B,u]\,[\kappa_A]$ in $\Aa(G)$ of an
  imprimitivity bimodule and a nondegenerate homomorphism.
\end{prop}

When we view $[X,u]$ as a morphism between objects
$(A,\alpha,\phi_A)$ and $(B,\beta,\phi_B)$ in the semi-comma category
$\Aa(G, (C_0(T),\rt))$, the composition
$\phi_{\K(X)}:=\kappa_A\circ\phi_A$ is an equivariant nondegenerate
homomorphism of $(C_0(T),\rt)$ into $(\K(X),\mu)$; since
$\kappa_A:A\to M(\K(X))$ trivially intertwines $\phi_A$ and
$\phi_{\K(X)}$,
$\kappa_A:(A,\alpha,\phi_A)\to(\K(X),\mu,\phi_{\K(X)})$ is a morphism
in the comma category $(C_0(T),\rt)\downarrow \Aa_{\nd}(G)$ considered
in \cite{kqrproper}. So every morphism in our semi-comma category
factors as the composition of a morphism implemented by an
imprimitivity bimodule and a morphism in the comma category of
\cite{kqrproper}. The imprimitivity bimodule has a left action of
$C_0(T)$ implemented by $\phi_{\K(X)}:C_0(T)\to M(\K(X))$ and a right
action of $C_0(T)$ implemented by $\phi_B:C_0(T)\to M(B)$, but we make
no assumption relating these left and right
actions.

\begin{example}
 To reinforce this last point, consider the morphism
  $[{}_AB_B,\beta]$ from $(A,\alpha,\phi_A)$ to $(B,\beta,\phi_B)$
  induced by an equivariant nondegenerate homomorphism $\psi:A\to
  M(B)$. Since $\K(B_B)=B$, the imprimitivity bimodule in the
  factorisation of $[{}_AB_B,\beta]$ is just ${}_BB_B$, but the left
  action of $C_0(T)$ is implemented by $\psi\circ \phi_A$ rather than $\phi_B$.
\end{example}

We summarise the above discussion:

\begin{cor}\label{factor}
  Every morphism $[X,u]:(A,\alpha,\phi_A)\to (B,\beta,\phi_B)$
  in the semi-comma category $\Aa(G, (C_0(T),\rt))$ factors as the
  composition $[{}_{\K(X)}X_B,u]\,[\kappa_A]$ of a morphism
  $[\kappa_A]$ coming from a morphism
  $\kappa_A:(A,\alpha,\phi_A)\to(\K(X),\mu,\phi_{\K(X)})$ in the comma
  category $(C_0(T),\rt)\downarrow \Aa_{\nd}(G)$ with a morphism
  $[{}_{\K(X)}X_B,u]$ implemented by a $\K(X)$\,--\,$B$ imprimitivity
  bimodule. We call this factorisation the \emph{canonical decomposition} of $[X,u]$.
\end{cor}

\begin{remark}\label{pleaseexplain}
  We now explain why we have added the homomorphism $\phi_A$ to
  the object $(A,\alpha)$ and then completely ignored it in our choice
  of morphisms. In particular, we explain why two obvious variations do not serve our purposes.

  First, we could have used objects in the comma category associated
  to the object $(C_0(T),\rt)$ of $\Aa(G)$.  As observed above, the
  homomorphism $\phi_A$ is there to ensure that our system is proper
  in the sense of \cite{proper}, which follows from Theorem~5.7 of
  \cite{integrable}. If we merely suppose that there is a morphism
  $[X,u]$ from $(C_0(T),\rt)$ to $(A,\alpha)$, so that $(A,\alpha,
  [X,u])$ is an object in the comma category, then we would need an
  analogue of \cite[Theorem~5.7]{integrable} which said that the
  existence of $(X,u)$ implied that $(A,\alpha)$ was proper. But there
  cannot be such a theorem --- indeed, for \emph{every} system
  $(A,\alpha)$ in $\Aa(G)$, proper or not, there is such a morphism
  from $(C_0(G),\rt)$ to $(A,\alpha)$. To see this, we take $(X,u)$ to
  be the bimodule $(A\otimes L^2(G),\alpha\otimes\rho)$ which
  implements the Morita equivalence between the second dual system
  $((A\rtimes_\alpha G)\rtimes_{\hat\alpha}G,\hat{\hat\alpha})$ and
  $(A,\alpha)$ (this is one formulation of the duality theorem of Imai
  and Takai; see, for example, \cite[Theorem~A.67]{enchilada}). Then
  with the left action of $C_0(G)$ given by $1\otimes M$, $(X,u)$
  defines a morphism in $\Aa(G)$ from $(C_0(G),\rt)$ to $(A,\alpha)$.

  Second, we could have used the morphisms from the usual comma
  category associated to the object $(C_0(T),\rt)$ in $\Aa(G)$, which
  are morphisms $[X,u]:(A,\alpha,\phi_A)\to (B,\beta,\phi_B)$ in
  $\Aa(G)$ such that $[X,\mu]\,[\phi_A]=[\phi_B]$. However, the
  right-Hilbert $B$-module underlying $[X,\mu]\,[\phi_A]$ is $X_B$, so
  $[X,\mu]\,[\phi_A]=[\phi_B]$ would imply in particular that $X_B$ is
  isomorphic to the trivial module $B_B$. Since we want to discuss
  naturality of Morita equivalences, we definitely want to allow
  morphisms based on non-trivial Hilbert modules. So we choose not to
  impose any relation between the left action of $C_0(T)$ on $X$
  coming from $\phi_A$ and the right action coming from $\phi_B$.
\end{remark}

\section{The functor Fix}\label{sec-fix}

We fix a free and proper action of a locally compact
group $G$ on a locally compact space $T$. In \cite{kqrproper}, it was
shown that Rieffel's generalised fixed-point algebra is the object map
in a functor $\Fix$ from the comma category $(C_0(T),\rt)\downarrow
\Aa_{\nd}(G)$ to $\CCnd$. In this section we extend $\Fix$ to a
functor on our semi-comma category. Since this category has the same
objects as the comma category $(C_0(T),\rt)\downarrow \Aa_{\nd}(G)$,
$\Fix$ is already defined on objects, and our first
main theorem says that we can extend the functor in \cite{kqrproper} to
cover our more general morphisms. We can then state our other main theorem on the
naturality of Rieffel's Morita equivalence.

We begin by recalling the construction of $\Fix$ from
\cite{kqrproper}. Suppose that $(A,\alpha,\phi)$ is an object in the
comma category $(C_0(T),\rt)\downarrow \Aa_{\nd}(G)$. We write $A_0$ for the dense $*$-subalgebra
$\phi(C_c(T))A\phi(C_c(T))$ of $A$; as in \cite{kqrproper}, we often
suppress the map $\phi$, so that
\[
A_0=C_c(T)AC_c(T)=\newspan\{fag:a\in A,\ f,g\in C_c(T)\}.
\]
(When $X$ is a $B$-module, the juxtaposition  $XB$
denotes $\newspan\{x\cdot b:x\in X,b\in B\}$.)

In \cite[\S2]{kqrproper} it was shown, using techniques of
Olesen-Pedersen \cite{OP1, OP2} and Quigg \cite{quigg, QR}, that for
every element $a\in A_0$, there is a multiplier $E(a)=E^A(a)$ such
that
\begin{equation}\label{defE}
  \omega(E^A(a))=\int \omega(\alpha_s(a))\,ds\ \text{ for all $\omega\in A^*$.}
\end{equation}
For fixed $f,g\in C_c(T)$, the map $a\mapsto E(fag)$ is norm
continuous on $M(A)$ \cite[Corollary~3.6(3)]{quigg}, and the range
$E(A_0)$ is a $*$-subalgebra of the fixed-point algebra $M(A)^\alpha$
\cite[Proposition~2.4]{kqrproper}. The algebra $\Fix(A,\alpha,\phi)$
is by definition the norm closure of $E(A_0)$ in $M(A)$.

Rieffel proved that the action $\alpha$ is proper and
saturated with respect to the subalgebra $A_0$ (see
\cite[Theorem~5.7]{integrable} and
\cite[Lemma~C.1]{aHRWproper2}). Thus there is a
generalised fixed-point algebra $A^\alpha$ in $M(A)$ which is Morita
equivalent to $A\rtimes_{\alpha,r}G$; an imprimitivity bimodule
$Z(A,\alpha,\phi_A)$ which implements this equivalence is obtained by
completing $A_0$ in the $A^\alpha$-valued inner
product $\langle a,b\rangle:=E^A(a^*b)$. In \cite[Proposition~3.1]{kqrproper}, it was shown that
$A^\alpha$ coincides with the fixed-point algebra $\Fix
A=\Fix(A,\alpha,\phi)$. It was also shown in \cite{kqrproper} that
$(A,\alpha,\phi)\mapsto\Fix(A,\alpha,\phi)$ is the object map in a
functor $\Fix:(C_0(T),\rt)\downarrow \Aa_{\nd}(G)\to \CCnd$:
indeed, the functor $\Fix$ simply sends the morphism based on a
nondegenerate homomorphism $\sigma:A\to M(B)$ to the restriction of
$\sigma$ to $\Fix A\subset M(A)$, which is itself is a nondegenerate
homomorphism of $\Fix A$ into $M(\Fix B)$
\cite[Proposition~2.6]{kqrproper}. The upshot is that, viewed as a
map on the comma category, Rieffel's assignment
$(A,\alpha,\phi)\mapsto A^\alpha$ becomes a functor.

To make $\Fix$ into a functor from our semi-comma category
$\Aa(G,(C_0(T),\rt))$ into $\CC$, we need to extend $\Fix$ to
morphisms $[X,u]$.  Our strategy, borrowed from \cite{enchilada}, is
to factor $[X,u]$ in $\Aa(G, (C_0(T),\rt))$ using
Corollary~\ref{factor}. We can apply the results of \cite{kqrproper}
to the nondegenerate homomorphism, and now we need to know how to $\Fix$ imprimitivity bimodules.

\begin{prop}\label{defgfpaib}
  Suppose that $(K,\mu,\phi_K)$ and $(B,\beta,\phi_B)$ are objects in
  the semi-comma category $\Aa(G, (C_0(T),\rt))$, and that
  ${}_{(K,\mu)}(X,u)_{(B,\beta)}$ is an imprimitivity bimodule. Denote
  by $L(u)$ the action of $G$ on the linking algebra $L(X)$, and let
  $\phi=\phi_K\oplus\phi_B$ be the diagonal embedding of $C_0(T)$ in $M(L(X))$.
  Then $(L(X),L(u),\phi)$ is an object in $\Aa(G, (C_0(T),\rt))$,
  and has generalised fixed-point algebra
  $L(X)^{L(u)}:=\Fix(L(X),L(u),\phi)$. The embeddings of $K$ and $B$ as corners in
  $L(X)$ extend to embeddings of $M(K)$ and $M(B)$ as corners
  $pM(L(X))p$ and $qM(L(X))q$ in $M(L(X))$, and these carry the generalised fixed-point
  algebras $K^\mu$ and $B^\beta$ onto the corners $pL(X)^{L(u)}p$ and
  $qL(X)^{L(u)}q$. With the operations inherited from $M(L(X))$, the
  corner $X^u:=pL(X)^{L(u)}q$ is a $K^\mu$\,--\,$B^\beta$ imprimitivity
  bimodule, and $L(X)^{L(u)}=L(X^u)$.
\end{prop}

\begin{proof}
  Let $k\in K_0$. It follows from \cite[Lemma~2.2]{kqrproper} that
  multiplying elements $\phi(f)$ by $\big(\begin{smallmatrix}
    E^K(k)&0\\0&0 \end{smallmatrix}\big)$ and
  $E^{L(X)}\big(\begin{smallmatrix} k&0\\0&0 \end{smallmatrix}\big)$
  both give $\big(\begin{smallmatrix} \int_G f\mu_s(k)\,
    ds&0\\0&0 \end{smallmatrix}\big)$; since $\phi$ is nondegenerate
  this implies that $\big(\begin{smallmatrix}
    E^K(k)&0\\0&0 \end{smallmatrix}\big)=E^{L(X)}\big(\begin{smallmatrix}
    k&0\\0&0 \end{smallmatrix}\big)$. Thus $K^\mu:=\clsp\{E^K(k):k\in
  K_0\}$ embeds as $pL(X)^{L(u)}p$. Similarly
  $B^\beta=qL(X)^{L(u)}q$. So $X^u:=pL(X)^{L(u)}q$ is a
  $K^\mu$\,--\,$B^\beta$ bimodule, and has $K^\mu$- and $B^\beta$-valued
  inner products given by computing in $L(X)^{L(u)}$.  We need to see
  that these inner products are full.

  Let $f,g\in C_c(T)$ and $b\in B$, so that $E^{L(X)}\big(\big(\begin{smallmatrix} 0&0\\0&fbg \end{smallmatrix}\big)\big)$ is a typical element of $qL(X)^{L(u)}q$.  Since $R\mapsto E^{L(X)}(fRg)$ is
  norm continuous on $M(L(X))$ and $p$ is full, we can
  approximate $E^{L(X)}\big(\big(\begin{smallmatrix} 0&0\\0&fbg \end{smallmatrix}\big)\big)$ by a
  sum $\sum_iE^{L(X)}(qR_ipS_iq) $ where $S_i,R_i\in L(X)_0$.  Since
  $\phi(C_c(T))E^{L(X)}(L(X)_0)$ is dense in $Z(L(X),L(u),\phi)$ by
  \cite[Lemma~3.4]{kqrproper}, we can approximate each $R_i^*\in
  L(X)_0$ by a sum $ \sum_j\phi(h_{ij})E^{L(X)}(R_{ij}) $ where
  $h_{ij}\in C_c(T)$ and $R_{ij}\in L(X)_0$. Now
  \begin{align}\label{eq-lem0}
    E^{L(X)}\big(\big(\begin{smallmatrix} 0&0\\0&fbg \end{smallmatrix}\big)\big)
    &\sim \sum_iE^{L(X)}(qR_ipS_iq)\notag\\
    &=\sum_i\langle pR_i^*q,pS_iq\rangle_{L(X)^{L(u)}}\notag\\
    &\sim\sum_{i,j}\big\langle p\phi(h_{ij})E^{L(X)}(R_{ij})q\,,\,pS_iq\big\rangle_{L(X)^{L(u)}}\notag\\
    &=\sum_{i,j} E^{L(X)}\big( qE^{L(X)}(R_{ij})^*\phi(h_{ij})^*pS_iq \big)\notag\\
    &=\sum_{i,j} qE^{L(X)}(R_{ij})^*pE^{L(X)}\big( \phi(h_{ij})^*S_i
    \big)q
  \end{align}
  because $p$ commutes with the image of $\phi$. But the
  right-hand side of \eqref{eq-lem0} is an element of
  $(X^u)^*X^u=\langle X^u\,,\, X^u\rangle_{B^\beta}$, and we have proved that the
  $B^\beta$-valued inner product is full.  Similarly, the
  $K^\mu$-valued inner product is full.
\end{proof}

We now extend the definition of $\Fix$ to right-Hilbert
bimodules. Suppose that $(X,u)$ is a right-Hilbert $A$\,--\,$B$ bimodule
which implements a morphism \[[X,u]:(A,\alpha,\phi_A)\to
(B,\beta,\phi_B)\] in $\Aa(G, (C_0(T),\rt))$. We factor
$[X,u]=[{}_{\K(X)}X_{B},u]\,[\kappa_A]$ as in
Corollary~\ref{factor}. By \cite[Proposition~2.6]{kqrproper}, the
extension of $\kappa_A$ to a homomorphism $\kappa_A:M(A)\to M(\K(X))$
restricts to a nondegenerate homomorphism
\[\kappa_A|:\Fix(A,\alpha,\phi_A)\to M(\Fix(\K(X),\mu,\kappa_A\circ
\phi_A)).
\]
Proposition~\ref{defgfpaib} implies that $X^u$ is a
$\Fix(\K(X),\mu,\kappa_A\circ \phi_A)$\,--\,$\Fix(B,\beta,\phi_B)$
imprimitivity bimodule, and since $\Fix(\K(X),\mu,\kappa_A\circ
\phi_A)=\K(X^u)$, we can view $\kappa_A|$ as a homomorphism of $\Fix
A$ into $\L(X^u)$. We define $\Fix(X,u)$ to be the right-Hilbert
$\Fix(A,\alpha,\phi_A)$\,--\,$\Fix(B,\beta,\phi_B)$ bimodule $X^u$ in
which the left action is given by $\kappa_A|$.  In view of
Lemma~\ref{lem-well-defined} below we can define
\[\label{defFixmor}
\Fix([X,u]):=[\Fix(X,u)]=[{}_{\K(X^u)}(X^u)_{\Fix(B,\beta,\phi_B)}]\,[\kappa_A|].
\]

\begin{lemma}\label{lem-well-defined}
  Suppose that $\psi:(X,u)\to (Y,v)$ is an equivariant isomorphism of
  right-Hilbert $A$\,--\,$B$ bimodules, so that $[X,u]=[Y,v]$ as morphisms
  from $(A,\alpha,\phi_A)$ to $(B,\beta,\phi_B)$. Then $\Fix(X,u)$ is
  isomorphic to $\Fix(Y,v)$.
\end{lemma}

\begin{proof}
  We factor $X$ and $Y$ as in Corollary~\ref{factor} to obtain actions
  $\mu$, $\mu'$ and nondegenerate homomorphisms $\kappa_A:(A,\alpha)\to
  (\K(X),\mu)$, $\kappa_A':(A,\alpha)\to (\K(Y),\mu')$.  Define
  $\rho:\K(X)\to \K(Y)$ by
  $\rho(\Theta_{x,w})=\Theta_{\psi(x),\psi(w)}$, and  then $\rho$ is a
  $\mu$\,--\,$\mu'$ equivariant isomorphism satisfying
  $\rho\circ\kappa_A=\kappa_A'$. It follows that
  $[\kappa_A|]=[\kappa_A'|]$, and we need to show that
  $({}_{\K(X)}X_B)^u$ and $({}_{\K(Y)}Y_B)^v$ are isomorphic as
  imprimitivity bimodules.  Define $\psi_L:L(X)\to L(Y)$ and
  $\phi_X:C_0(T)\to M(L(X))$, $\phi_Y:C_0(T)\to M(L(Y))$ by
  \[
\psi_L=
  \begin{pmatrix}\rho&\psi\\ * &\id_B
  \end{pmatrix},\quad \phi_{L(X)}
  =(\kappa_A\circ\phi_A)\oplus\phi_B\quad\text{and}\quad\phi_{L(Y)}=
(\kappa_A'\circ\phi_A)\oplus\phi_B.
  \]
  Then $\psi_L$ is an isomorphism which maps $L(X)_0$ to $L(Y)_0$, and
  hence $\psi_L|$ is an isomorphism of the generalised fixed point
  algebras $L(X)^{L(u)}$ and $L(Y)^{L(v)}$. It follows that the three
  corners $\rho|$, $\psi|$ and $\id_B|$ in $\psi_L|$ form an
  imprimitivity-bimodule isomorphism of $({}_{\K(X)}X_B)^u$ onto
  $({}_{\K(Y)}Y_B)^v$.
\end{proof}

\begin{thm}\label{thm-fix}
  Suppose a locally compact group $G$ acts freely and properly on a
  locally compact space $T$. Then the assignments
  \[
  (A,\alpha,\phi_A)\mapsto \Fix (A,\alpha,\phi_A)\ \text{ and }\
  [X,u]\mapsto \Fix([X,u])
  \]
  form a functor $\Fix:\Aa(G,(C_0(T),\rt))\to\CC$.
\end{thm}

Theorem~\ref{thm-fix} is the hardest result in this paper: it includes, for example, the assertion (proved as
Theorem~\ref{thm-lemma-b} in \S\ref{2ibs}) that if
$_{(A,\alpha)}(X,u)_{(B,\beta)}$ and $_{(B,\beta)}(Y,v)_{(C,\gamma)}$
are imprimitivity bimodules implementing isomorphisms in the
semi-comma category, then $(X\otimes_B Y)^{u\otimes v}$ is isomorphic
to $X^u\otimes_{B^\beta}Y^v$. We prove Theorem~\ref{thm-fix} in
\S\ref{pf-fixed}.

\begin{prop}\label{prop-rcp}
  Suppose a locally compact group $G$ acts freely and properly on a
  locally compact space $T$.  Then the assignments
  $(A,\alpha,\phi_A)\mapsto A\rtimes_{\alpha,r}G$ and $(X,u)\mapsto
  X\rtimes_{u,r}G$ give a functor $\RCP$ from $\Aa(G,(C_0(T),\rt))$ to
  $\CC$.
\end{prop}

\begin{proof}
  It follows from \cite[Theorem~3.7]{enchilada} that the assignments
  $(A,\alpha)\mapsto A\rtimes_{\alpha,r}G$ and $(X,u)\mapsto
  X\rtimes_{u,r}G$ give a functor from $\Aa(G)$ to $\CC$ (since
  $\Aa(G)$ and $\CC$ are subcategories of the categories considered in
  \cite[Theorem~3.7]{enchilada}).  After restricting to the semi-comma
  category we still have a functor.
\end{proof}

We can now state our main theorem, which extends \cite[Theorem~3.2]{kqrproper} to our
more general setting. We will prove it in \S\ref{pfthmmain}.

\begin{thm}\label{thm-main}
  Suppose a locally compact group $G$ acts freely and properly on a
  locally compact space $T$. Then Rieffel's Morita equivalences
  \[
  Z(A,\alpha,\phi):A\rtimes_{\alpha,r}G\to \Fix(A,\alpha,\phi)
  \]
  form a natural isomorphism between the functors
  $\RCP:\Aa(G,(C_0(T),\rt))\to \CC$ of Proposition~\ref{prop-rcp} and
  $\Fix:\Aa(G,(C_0(T),\rt))\to \CC$ of Theorem~\ref{thm-fix}.
\end{thm}

\section{Proofs of the main theorems}\label{sec-proofs}

The main problem in proving Theorem~\ref{thm-fix} is to prove that $\Fix$ respects composition. To help
understand the issues which arise, we start the proof. Suppose that
${}_{(A,\alpha)}(X,u)_{(B,\beta)}$ and
${}_{(B,\beta)}(Y,v)_{(C,\gamma)}$ are right-Hilbert bimodules. We
need to prove that $\Fix(X\otimes_BY,u\otimes v)$ and
$\Fix(X,u)\otimes_{B^\beta}\Fix(Y,v)$ are isomorphic as right-Hilbert
$A^\alpha$\,--\,$C^\gamma$ bimodules. To compute
$\Fix(X,u)\otimes_{B^\beta}\Fix(Y,v)$, we factor $(X,u)$ and $(Y,v)$
as in Corollary~\ref{factor}, and apply $\Fix$ to each
factor. Proposition~2.6 of \cite{kqrproper} gives nondegenerate
homomorphisms $\kappa_A|$ and $\kappa_B|$,
Proposition~\ref{defgfpaib} gives imprimitivity bimodules $X^u$ and
$Y^v$, and then $\Fix(X,u)\otimes_{B^\beta}\Fix(Y,v)$ is  the composition
\begin{equation}\label{eqfunct1}
  \xymatrix{
    A^\alpha
    \ar[r]^-{\kappa_A|}
    &\K(X)^\mu
    \ar[r]^-{X^u}
    &B^\beta
    \ar[r]^-{\kappa_B|}
    &\K(Y)^\notnu
    \ar[r]^-{Y^v}
    &C^\gamma.
  }
\end{equation}
To compare this to the canonical form of $\Fix(X\otimes_B Y,u\otimes v)$, we
need to first realise $X^u\otimes_{B^\beta}\K(Y)^\notnu$ in canonical
form, which we do in \S\ref{hmib}. This gives us a sequence with two
homomorphisms on the left, which compose easily, and two imprimitivity
bimodules $Z^w$, say, and $Y^v$ on the right. We then deal with tensor
products of the form $Z^w\otimes_{K^\mu}Y^\notnu$ in \S\ref{2ibs}. The
proof of Theorem~\ref{thm-fix} itself is in \S\ref{pf-fixed}.

\subsection{Composing a homomorphism with an imprimitivity
  bimodule}\label{hmib} The object of this subsection is to prove the following proposition.

\begin{prop}
  \label{prop-lemma-a}
  Suppose that $(K,\mu,\phi_K)$, $(B,\beta,\phi_B)$ and $(C,\gamma,
  \phi_C)$ are objects in the semi-comma category $\Aa(G,
  (C_0(T),\rt))$, that ${}_{(K,\mu)}(X,u)_{(B,\beta)}$ is an
  imprimitivity bimodule, and that $\psi:(B,\beta)\to
  M(C,\gamma)$ is a nondegenerate homomorphism satisfying $\psi\circ\phi_B=\phi_C$.  Then
  \begin{equation*}
    X^{u}\tensor_{B^{\beta}}C^{\gamma}\cong (X\tensor_{B}C)^{u\tensor\gamma}
  \end{equation*}
  as right-Hilbert $K^\mu$\,--\,$ C^{\gamma}$ bimodules.
\end{prop}

We begin the proof of Proposition~\ref{prop-lemma-a} with two lemmas.

\begin{lemma}
  \label{lem-basic}
  Let $X$ be a right Hilbert $B$-module and $Y$ a
  right-Hilbert $B$\,--\,$C$ bimodule.  For $x\in X$, define the \emph{creation
  operator} $C_{x}:Y\to X\otimes_{B}Y$ by $C_{x}(y):=x\otimes y$ and the \emph{annihilation operator}
  $A_{x}:X\otimes_{B}Y\to Y$ by
  $A_{x}(x'\otimes y):=\rip B<x,x'>\cdot y$.  Then
  $C_{x}$ is adjointable with adjoint $C_{x}^{*}=A_{x}\in
  \mathcal{L}(X\tensor_{B}Y,Y)$, and $\|C_{x}\|=\|A_{x}\|\le
  \|x\|$ for all $x\in X$, with equality if $\kappa_B:B\to \mathcal{L}(Y)$ is injective.
\end{lemma}
\begin{proof}
For $x,x'\in X$ and $y,y'\in Y$ we have
  \begin{equation*}
    \btrip C<C_{x}(y),x'\tensor y'> = \brip C<y,{\rip B<x,x'>}\cdot
    y'> = \brip C<y,A_{x}(x'\tensor y')>,
  \end{equation*}
which implies that $C_x$ is adjointable with $C_{x}^{*}=A_{x}$. Then $C_{x}^{*}C_{x}(y)=\rip B<x,x>\cdot
  y$, so $C_x^*C_x=\kappa_B(\rip B<x,x>)$, and since $\|C_x\|^2=\|C_x^*C_x\|$, the result follows.
\end{proof}

\begin{lemma}
  \label{lem-lemma-c} Let $X$ be a right Hilbert $B$-module and $Y$ a
  right Hilbert $C$-module, and let $p_X, q_X, p_Y, q_Y$ be the corner
  projections in $M(L(X))$ and $M(L(Y))$,
  respectively.  Suppose that $\Psi:L(X)\to M(L(Y))$ is a
  nondegenerate homomorphism satisfying $\Psi(p_X)=p_Y$ and
  $\Psi(q_X)=q_Y$. Then there are nondegenerate homomorphisms
  $\Psi_{\K}:\K(X)\to \L(Y)$ and $\Psi_B:B\to M(C)$, and a
  $\Psi_{\K}$\,--\,$\Psi_B$ compatible linear map $\Psi_X:X\to M(Y)$,
  such that
  \[
  \Psi\begin{pmatrix}
    k&x\\ *&b
  \end{pmatrix}=
  \begin{pmatrix}
    \Psi_{\K}(k)&\Psi_X(x)\\ *&\Psi_B(b)
  \end{pmatrix}.
  \]
  The map $\Omega(x\tensor c):=\Psi_X(x)\cdot c$ extends to an
  isomorphism $\Omega: X\tensor_{B}C\to Y$ of right-Hilbert
  $\K(X)$\,--\,$C$ bimodules.
\end{lemma}

\begin{proof}
The existence of $\Psi_{\K}$, $\Psi_B$ and $\Psi_X$ is proved in \cite[Lemma~1.52]{enchilada}; since $\Psi$ is nondegenerate,
  \cite[Lemma~1.52]{enchilada} also implies that $\overline{\Psi_X(X)\cdot C}=Y$, and thus $\Omega$ has dense
  range.  For $x_1,x_2\in X$ and $c_1,c_2\in C$ we have
  \begin{align*}
    \trip C< x_1\otimes c_1,x_2\otimes c_2>
    &=\langle\langle x_2\,,\, x_1\rangle_B\cdot c_1\,,\, c_2\rangle_C=\langle \Psi_B(\langle x_2\,,\, x_1\rangle_B)c_1\,,\, c_2\rangle_C\\
    &=\langle\langle\Psi_X(x_2)\,,\,\Psi_X(x_1)\rangle_Cc_1\,,\, c_2\rangle_C=c_1^*\langle\Psi_X(x_2)\,,\,\Psi_X(x_1)\rangle_C^*c_2\\
    &=\langle\Psi_X(x_1)\cdot c_1\,,\,\Psi_X(x_2)\cdot c_2\rangle_C,
  \end{align*}
  so $\Omega$ is inner-product preserving. It is clearly a $C$-module homomorphism, and is a $\K(X)$-module homomorphism because $\Psi_X(k\cdot x)=\Psi_{\K}(k)\cdot\Psi_X(x)$.
  \end{proof}

\begin{remark}\label{idaction}
In the proof of Proposition~\ref{prop-lemma-a} we make a lot of identifications, and it will be helpful to get these straight before we begin. Suppose $(Z,v)$ is a right Hilbert $(C,\gamma)$-module. Then its linking algebra $L(Z)$ carries an action $L(v)$; in the top left-hand corner of $L(v)$ is the action $\mu$ on $\K(Z)$, which can be described on an operator $k\in \K(Z)$ as $\mu_s(k)=v_s\circ k\circ v_s^{-1}$ (see the top of page 292 in \cite{combes}). When we identify $L(Z)$ with $\K(Z\oplus C)$, the action $L(v)$ is given by
\[
L(v)_s(K)=(v_s\oplus \gamma_s)\circ K\circ(v_s^{-1}\oplus\gamma_s^{-1})\ \text{ for $K\in \K(Z\oplus C)$.}
\]
The automorphisms $L(v)_s$ extend to multipliers $m\in M(\K(Z\oplus C))$ according to the formula $L(v)_s(m)K=L(v)_s(mL(v)_s^{-1}(K))$, and when we identify $\L(Z\oplus C)$ with the multiplier algebra of $\K(Z\oplus C)=L(Z)$, $L(v)$ is given by
\[
L(v)_s(R)=(v_s\oplus \gamma_s)\circ R\circ(v_s^{-1}\oplus\gamma_s^{-1})\ \text{ for $R\in \L(Z\oplus C)$.}
\]
On elements $T$ in the corner $\L(C,Z)=p\L(Z\oplus C)q$, for example, $L(v)_s$ is given by
\begin{align*}
L(v)_s\big(\big(\begin{smallmatrix}0&T\\0&0\end{smallmatrix}\big)\big)
=(v_s\oplus \gamma_s)\circ \big(\begin{smallmatrix}0&T\\0&0\end{smallmatrix}\big)\circ(v_s^{-1}\oplus\gamma_s^{-1}) =\big(\begin{smallmatrix}0&v_s\circ T\circ\gamma_s^{-1}\\0&0\end{smallmatrix}\big).
\end{align*}
\end{remark}

\begin{proof}
  [Proof of Proposition~\ref{prop-lemma-a}]
  We will define $\Phi_{L}:L(X_{B})\to M\bigl(L(X\tensor_{B}C)\bigr)$ by
  identifying $M\bigl(L(X\tensor_{B}C)\bigr)$ with
\[
  \mathcal{L}((X\tensor_{B}C)\oplus C)=\begin{pmatrix}\L(X\otimes_B C)&\L(C,X\otimes_BC)\\
  \L(X\otimes_BC,C)&\L(C)\end{pmatrix},
  \]
as in Remark~\ref{idaction} with $Z=X\otimes_B C$ and $v=u\oplus\gamma$, and defining $\Phi_L=(\Phi_L^{ij})$ on the corners.
We will then show that $\Phi_L$ is  a morphism in the comma
  category $(C_0(T),\rt)\downarrow\Aa_{\nd}(G)$, and apply \cite[Proposition~2.6]{kqrproper} to $\Phi_L$.

  For $a\in K, b\in B, c\in C$ and $x\in X$ define
  \begin{align*}
    &\Phi_{L}^{11}:K\to \L(X\otimes_B C)\ \text{ by }\ \Phi_{L}^{11}(a)(x\tensor c)=a\cdot x\tensor c,\\
    &\Phi_{L}^{12}: X\to \L(C,X\otimes_B C)\ \text{ by }\ \Phi_{L}^{12}(x)=C_{x},\\
    &\Phi_{L}^{21}:\widetilde X\to\L(X\otimes_B C, C)\ \text{ by }\ \Phi_{L}^{21}(\flat(x))=A_{x},\ \text{ and}\\
    &\Phi_{L}^{22}:B\to\L (C)\ \text{ by }\ \Phi_{L}^{22}(b)(c)=\psi(b)(c),
  \end{align*}
  where $C_x$ and $A_x$ are the creation and annihilation
  operators of Lemma~\ref{lem-basic}.  To see that $\Phi_{L}$ is
  a homomorphism, we check what happens on the four corners.
  Let
  \[
  m_1=\begin{pmatrix}
    a_1&x_1\\\flat(y_1)&b_1 \end{pmatrix}\quad\text{and}\quad
  m_2=\begin{pmatrix} a_2&x_2\\\flat(y_2)&b_2 \end{pmatrix}\in L(X).
  \]

The top left-hand corner of $\Phi_L(m_1)\circ\Phi_L(m_2)$ applied to $x_3\otimes c\in X\otimes_B C$ is
  \begin{align*}
    \big(\Phi_L^{11}(a_1)\circ \Phi_L^{11}(a_2)+\Phi_L^{12}(x_1)
    \circ&\Phi_L^{21}(\flat(y_2))\big)(x_3\otimes
    c)\\
    &=(a_1a_2)\cdot x_3\otimes c+C_{x_1}(A_{y_2}(x_3\otimes c))\\
    &=(a_1a_2)\cdot x_3\otimes c+x_1\otimes\langle y_2\,,\,
    x_3\rangle_B\cdot c\\
    &=(a_1a_2)\cdot x_3\otimes c+x_1\cdot \langle y_2\,,\,
    x_3\rangle_B\otimes c\\
    &=(a_1a_2)\cdot x_3\otimes c+{}_K\langle x_1\,,\, y_2\rangle\cdot
    x_3\otimes c\\
    &=\Phi_L^{11}(a_1a_2+{}_K\langle x_1\,,\,
    y_2\rangle)(x_3\otimes c),
  \end{align*}
  which is the top left-hand corner of $\Phi_L(m_1m_2)$ applied to $x_3\otimes c$.
The top right-hand corner of $\Phi_L(m_1)\circ \Phi_L(m_2)$ applied to $c\in C$ is
  \begin{align*}
    \big(
    \Phi_L^{11}(a_1)\circ\Phi_L^{12}(x_2)+\Phi_L^{12}(x_1)\circ\Phi_L^{22}(b_2)
    \big)(c)
    &=\Phi_L^{11}(a_1)(x_2\otimes c)+\Phi_L^{12}(x_1)(\psi(b_2)c)\\
    &=a_1\cdot x_2\otimes c+x_1\otimes\psi(b_2)c\\
    &=a_1\cdot x_2\otimes c+x_1\cdot b_2\otimes c\\
    &=\Phi_L^{12}(a_1\cdot x_2+x_1\cdot b_2)(c),
  \end{align*}
which is the top right-hand corner of $\Phi_L(m_1m_2)$ applied to $c$.
  The calculations for the two bottom corners are similar. So $\Phi_L$
  is a homomorphism.  Since $\Phi_{L}^{12}(X)(C)$ spans a dense subset
  of $X\tensor_{B}C$, it follows from \cite[Lemma~1.52]{enchilada}
  that $\Phi_{L}$ is nondegenerate.

  To show that $\Phi_{L}$ is $L(u)$\,--\,$ L(u\tensor
  \gamma)$ equivariant, we recall from Remark~\ref{idaction} that $L(u\otimes\gamma)_s$ is conjugation by $(u_s\otimes\gamma_s)\oplus \gamma_s$. Thus the calculations
  \begin{align*}
    \big(\Phi_{L}^{11}\bigl(\mu_{s}(a)\big)\bigr)(x\otimes c)&=
    \mu_{s}(a)\cdot x\tensor
    c
    = u_{s}\tensor \gamma_{s}\bigl(a\cdot
    u_{s}^{-1}\tensor\gamma_{s}^{-1}(x\tensor c)\bigr)\\
    &= u_{s}\tensor \gamma_{s}\bigl(\Phi_L^{11}(a)(
    u_{s}^{-1}\tensor\gamma_{s}^{-1}(x\tensor c))\bigr),\\
  \Phi_{L}^{12}\bigl(u_{s}(x)\bigr)(c)&=u_{s}(x)\tensor c \notag = u_{s}\tensor \gamma_{s}\bigl(x\tensor\gamma_{s}^{-1}(c)\bigr)\\&=u_{s}\tensor \gamma_{s}\bigl(\Phi_L^{12}(x)(\gamma_{s}^{-1}(c))\bigr),\text{ and }\\
 \Phi_L^{22}(\beta_s(b))(c)&=\psi(\beta_{s}(b))c= \gamma_{s}\bigl(\psi(b)\bigr)c= \gamma_{s}\bigl(\Phi_L^{22}(b)\gamma_s^{-1}(c)\bigr)
  \end{align*}
  show that three of the four corners are appropriately equivariant. The fourth is too:
  \begin{align*}
  \Phi_L^{21}(\flat&(u_s(x)))(x_{1}\tensor c_{1})
   = A_{u_{s}(x)}(x_{1}\tensor c_{1})
   = \psi\bigl(\brip B<u_{s}(x),x_{1}>\bigr)c_{1}\\
   &= \psi\bigl(\beta_{s}\bigl(\brip
    B<x,u_{s}^{-1}(x_{1})>\bigr)\bigr)c_{1}
    = \gamma_{s}\bigl(\psi\bigl(\brip
    B<x,u_{s}^{-1}(x_{1})>\bigr)\gamma_{s}^{-1}(c_{1})\bigr)\\
    &=
    \gamma_{s}
    \bigl(\brip B<x,u_{s}^{-1}(x_{1})>\cdot \gamma_{s}^{-1}(c_{1})\bigr)
    =\gamma_{s}\bigl(\Phi_{L}^{21}\bigl(\flat(x)\bigr)\bigl(u_{s}^{-1}
    \tensor \gamma_{s}^{-1}(x_{1}\tensor c_{1})\bigr)\bigr).
  \end{align*}
   Thus $\Phi_{L}$ is
  $L(u)$\,--\,$ L(u\tensor \gamma)$-equivariant.

Define $\phi_{L(X)}:C_0(T)\to M(L(X))$ and
  $\phi_{L(X\otimes_BC)}:C_0(T)\to M(L(X\otimes_BC))$ by
$\phi_{L(X)}=
    \phi_K\oplus \phi_B$ and $\phi_{L(X\otimes_BC)}=\phi_K\oplus \phi_C$. Then $\Phi_{L}$ is $C_{0}(T)$-linear because $\psi$ is, and we can apply \cite[Proposition~2.6]{kqrproper} to $\Phi_{L}$. This gives a nondegenerate homomorphism
  \begin{equation*}
    \Phi_{L}|:L(X)^{L(u)}\to M\big(L(X\tensor_{B}C)^{L(u\tensor\gamma)}\big).
  \end{equation*}
 By Proposition~\ref{defgfpaib}
  the two fixed-point algebras are themselves linking algebras:
$L(X)^{L(u)}=L(X^u)$ and $L(X\tensor_{B}C)^{L(u\otimes\gamma)}=L((X\otimes_BC)^{u\otimes\gamma})$.
   Now applying Lemma~\ref{lem-lemma-c} to $\Phi_L|:L(X^u)\to M\big(L((X\tensor_{B}C)^{u\tensor\gamma})\big)$ gives the required  isomorphism.
\end{proof}

\subsection{Composing two imprimitivity bimodules}\label{2ibs}

\begin{theorem}
  \label{thm-lemma-b}
  Suppose that $_{(A,\alpha)}(X,u)_{(B,\beta)}$ and
  $_{(B,\beta)}(Y,v)_{(C,\gamma)}$ are \ib s.  Then
  $X^{u}\tensor_{B^{\beta}}Y^{v}\cong (X\tensor_{B}Y)^{u\tensor v}$ as
  $A^{\alpha}$\,--\,$ C^{\gamma}$ \ib s.
\end{theorem}

A key ingredient in the proof of Theorem~\ref{thm-lemma-b} is
the following $3\times 3$ matrix trick.

\begin{lemma}
  \label{lem-lemma-e}
  Suppose that $D$ is a $C^{*}$-algebra and that $p_1$, $p_2$ and $p_3$ are
  projections in $M(D)$ such that $p_{1}+p_{2}+p_{3}=1$ and each $p_iDp_j$ is a $p_iDp_i$\,--\,$ p_jDp_j$ imprimitivity
  bimodule. Then the map
  \begin{equation}\label{eq:33}
    p_{1}dp_{2}\tensor p_{2}d'p_{3}\mapsto p_{1}dp_{2}d'p_{3}
  \end{equation}
  extends to a $p_1Dp_1$\,--\,$ p_3Dp_3$ imprimitivity-bimodule
  isomorphism of $p_{1}Dp_{2}\tensor_{p_{2}Dp_{2}}p_{2}Dp_{3}$ onto
  $p_{1}Dp_{3}$.
\end{lemma}
\begin{proof}
  If $x,x_{1}\in p_{1}Dp_{2}$ and $y,y_{1}\in p_{2}Dp_{3}$, then
  \begin{align*}
    \trip p_{3}Dp_{3}<x\tensor y,x_{1}\tensor y_{1}>
    &= \brip p_{3}Dp_{3}<{\rip p_{2}Dp_{2}<x_{1},x>}y,y_{1}>
    = \rip p_{3}Dp_{3}<x^{*}_{1}xy,y_{1}>\\
    &= y^{*}x^{*}x_{1}y_{1}=\rip p_{3}Dp_{3}<xy,x_{1}y_{1}>.
  \end{align*}
  Therefore \eqref{eq:33} defines a right-Hilbert bimodule isomorphism
  of $p_{1}Dp_{2}\tensor_{p_{2}Dp_{2}}p_{2}Dp_{3}$ into $p_{1}Dp_{3}$.
  This map clearly preserves the right and left actions, and
  another computation like that above shows that it preserves the left
  inner products as well.  Therefore \eqref{eq:33} is an imprimitivity
  bimodule isomorphism of
  $p_{1}Dp_{2}\tensor_{p_{2}Dp_{2}}p_{2}Dp_{3}$ onto a closed
  sub-bimodule $M$ of $p_{1}Dp_{3}$.  Since, for example, $p_3Dp_1Dp_3$ is dense in $p_3Dp_3$ because $p_3Dp_1$ is an imprimitivity bimodule,  the inner products of $M$ are full. So the  Rieffel correspondence \cite[Theorem~3.22]{tfb} implies that
  the map \eqref{eq:33} is surjective.
\end{proof}

\begin{remark}
  \label{rem-new-3x3}
  We think the $3\times 3$-matrix trick in Lemma~\ref{lem-lemma-e} may be new, and it yields an
  interesting generalisation of \cite[Proposition~1.48]{enchilada}.  In the lemma, we assume that
  $p_{i}Dp_{j}Dp_{i}$ is dense in $p_{i}Dp_{i}$, and the proof shows
  that we then have $p_{i}Dp_{k}Dp_{j}$  dense in $p_{i}Dp_{j}$. Then, since
  $p_{1}+p_{2}+p_{3}=1$,
  \begin{equation*}
    \overline{D
      p_{k}D}=\sum_{i,j}p_{i}\overline{Dp_{k}D}p_{j}=\sum_{i,j}p_{i}Dp_{j}
    = D,
  \end{equation*}
 and each $p_{k}$ is full.  Conversely, if we know that each
  $p_{k}$ is full, then the hypotheses of the lemma are satisfied.  So
the off-diagonal entries are \ib s if and only
  if $p_{1}$, $p_{2}$ and $p_{3}$ are full.
\end{remark}

Next we prove some basic facts about adjointable operators on tensor products of Hilbert modules.
If $Z$ and $X$ are Hilbert $B$-modules and  $Y$ is a $B$\,--\,$C$ right-Hilbert
bimodule, then $T\mapsto T\tensor_{B}1$ induces a norm-decreasing map
$\tau_{Y}$ from $\mathcal{L}(X,Z)$ to $\mathcal{L}(X\tensor_{B}Y,
Z\tensor_{B}Y)$ (see, for example, \cite[Lemma~I.3]{tfb^2}).
\begin{remark}
  \label{rem-injective}
  If $\kappa_B:B\to \mathcal{L}(Y)$ is injective, then $\tau_{Y}$
  is injective.  To see this, suppose $T\in \mathcal{L}(X,Z)$ is nonzero. Then there  exists $x\in X$
  such that $Tx\not=0$.  Then $\langle Tx,Tx\rangle_B^{1/2}\not=0$, and since $\kappa_B$ is injective, there exists  $y\in Y$ such that $\langle Tx,Tx\rangle_B^{1/2}\cdot
  y\not=0$.  But then $\tau_{Y}(T)(x\tensor y)=Tx\tensor
  y\not=0$.
\end{remark}

For the proof of Lemma~\ref{lem-main}, we need a slight generalisation
of \cite[Proposition~4.7]{lan:hilbert}.

\begin{lemma}
  \label{lem-helper}
  Suppose that $X$ and $Z$ are right Hilbert $B$-modules and that $Y$
  is a $B$\,--\,$ C$ \ib.  Then $\tau_{Y}:T\mapsto T\otimes 1$ is an isometric isomorphism of
  $\mathcal{L}(X,Z)$ onto $\mathcal{L}(X\tensor _{B}Y,Z\tensor_{B}Y)$
  which takes $\mathcal{K}(X,Z)$ onto
  $\mathcal{K}(X\tensor_{B}Y,Z\tensor_{B} Y)$.
\end{lemma}

\begin{proof}
  Let $\tau_{\widetilde Y}$ be the norm-decreasing map $S\mapsto S\otimes_C1$ of
  $\mathcal{L}(X\tensor_{B}Y,Z\tensor_{B}Y)$ into
  $\mathcal{L}(X\tensor_{B}Y\tensor _{C}\widetilde Y,
  Z\tensor_{B}Y\tensor_{C}\widetilde Y)$.  Since $Y\tensor_{C}\widetilde Y\cong B$, we have an
  isometric isomorphism  $\psi$ of $\mathcal{L}(X\tensor_{B}Y\tensor
  _{C}\widetilde Y, Z\tensor_{B}Y\tensor_{C}\widetilde Y)$ onto
  $\mathcal{L}(X,Z)$.  It is straightforward to see that $\psi\circ
  \tau_{\widetilde Y}\circ \tau_{Y}=\id$. Since $\tau_{\widetilde Y}$ is injective so is $\psi\circ
  \tau_{\widetilde Y}$, and hence $\psi\circ
  \tau_{\widetilde Y}\circ \tau_{Y}=\id$ implies $\tau_{Y}\circ \psi\circ
  \tau_{\widetilde Y}=\id$.  Thus $\tau_Y$ is a bijection.  Since $\tau_Y$ is norm decreasing, $\psi\circ
  \tau_{\widetilde Y}\circ \tau_{Y}=\id$ implies that $\tau_Y$ is isometric.

 For $x\in X$, $x\in Z$ and $b\in B$, a calculation shows that $\tau_Y(\Theta_{z\cdot b, x})=C_z\circ \kappa_B(b)\circ A_x$, which is compact because $\kappa_B(b)$ is.   Since every element of $Z$ has the form $z\cdot b$ (see, for example, \cite[Proposition~2.3]{tfb}), we deduce that $\tau_Y$ maps $\mathcal{K}(X,Z)$ into $\mathcal{K}(X\tensor_{B}Y,Z\tensor_{B} Y)$. The surjectivity argument used above shows that $\tau_Y$ maps the compacts onto the compacts.
\end{proof}

\begin{lemma}
  \label{lem-main}
  Suppose that $X$ is a right Hilbert $B$-module and that $Y$ is a
  $B$\,--\,$ C$ \ib.  Then Lemma~\ref{lem-basic} applies, and
  \begin{enumerate}
  \item the map $x\mapsto C_{x}$ is an isometric isomorphism of $X$
    onto $\mathcal{K}(Y,X\tensor_{B}Y)$ and
  \item the map $x\mapsto A_{x}$ is an isometric isomorphism of
    $\widetilde X$ onto $\mathcal{K}(X\tensor_{B}Y,Y)$.
  \end{enumerate}
\end{lemma}

\begin{proof}
  The map $L_{x}:B\to X$ given by $L_{x}(b):=x\cdot b$ is in
  $\mathcal{K}(B,X)$ and $x\mapsto L_{x}$ is an isometric isomorphism
  of $X$ onto $\mathcal{K}(B,X)$ (see, for example,
  \cite[Lemma~2.32]{tfb}).  Using Lemma~\ref{lem-helper} and the isomorphism $B\otimes_BY\to Y$ given by $b\otimes y\mapsto b\cdot y$ we have
  isometric isomorphisms
  \begin{equation}\label{decompC}
      \xymatrix{{X}\ar[r]&\mathcal{K}(B,X)\ar[r]&
      \mathcal{K}(B\tensor_{B}Y,X\tensor_{B}Y)
      \ar[r]^-{\psi}&\mathcal{K}(Y,X\tensor_{B}Y)   }
  \end{equation}
 which take
$x\mapsto L_{x}\mapsto L_{x}\tensor 1\mapsto \psi(L_x\otimes 1).
$
 For $b\in B$ and $y\in Y$ we have
 \begin{align*}
 \psi(L_x\otimes 1)(b\cdot y)&=(L_x\otimes 1)(b\otimes y)=x\cdot b\otimes y= x\otimes b\cdot y=C_x(b\cdot y).
 \end{align*}
 Since $BY$ is dense in $Y$ we deduce that $\psi(L_x\otimes 1)=C_x$, and this gives (1).

  To establish (2), we observe that $\flat(x)\mapsto a(x)$ given
  by $a(x)(x')=\langle x,x'\rangle_B$ is an isometric isomorphism of
  $\widetilde X$ onto $\mathcal{K}(X,B)$ (see, for example,
  \cite[Lemma~2.32]{tfb}).  Applying Lemma~\ref{lem-helper}, we have
  isometric isomorphisms
  \begin{equation*}
    \xymatrix{{\widetilde{X}}
  \ar[r]& \mathcal{K}(X,B)\ar[r]&
      \mathcal{K}(X\tensor_{B}Y,
      B\tensor_{B}Y)\ar[r]^-{\psi'}&\mathcal{K}(X\tensor_{B}Y,Y)}
  \end{equation*}
  which takes $\flat(x)\mapsto a(x)\mapsto a(x)\tensor 1\mapsto \psi'(a(x)\otimes 1)$.
 Since $(a(x)\otimes 1)(x'\otimes y)=\langle x,x'\rangle_B\otimes y$ we have $\psi'(a(x)\otimes 1)(x'\otimes y)=\langle x, x'\rangle_B\cdot y=A_x(x'\otimes y)$.  Thus $\psi'(a(x)\otimes 1)=A_x$, and (2) follows.
\end{proof}

\begin{proof}[Proof of Theorem~\ref{thm-lemma-b}]
We view $(X\tensor_{B}Y)\oplus Y\oplus C$ as a right Hilbert
  $C$-module, and realise
  $\K\bigl((X\tensor_{B}Y)\oplus Y\oplus C\bigr)$ as matrices
  \begin{equation*}
    \begin{pmatrix}
      \K(X\tensor_{B}Y)&\K(Y,X\tensor_{B}Y)& \K(C,X\tensor_{B}Y) \\
      \K(X\tensor_{B}Y,Y)&\K(Y)&\K(C,Y) \\
      \K(X\tensor_{B}Y,C)&\K(Y,C)&\K(C)
    \end{pmatrix}.
  \end{equation*}
We identify the diagonal entries with $A$, $B$ and $C$,
  respectively, and we use Lemma~\ref{lem-main} to identify the $(1,2)$ entry with $X$ and the $(2,1)$ entry with $\widetilde X$. We also know that $L_y:c\mapsto y\cdot c$ gives an isomorphism $L$ of $Y$ onto $\K(C,Y)$. Thus we can realise $\K\bigl((X\tensor_{B}Y)\oplus
  Y\oplus C\bigr)$ as
  \begin{equation*}
   F:= \begin{pmatrix}
      A&X&X\tensor_{B}Y\\
      \widetilde{X}&B&Y \\
      (X\tensor_{B}Y)^{\sim}&\widetilde{Y}&C
    \end{pmatrix};
  \end{equation*}
 the action on $F$ on $(X\tensor_{B}Y)\oplus Y\oplus C$ is given by
  \begin{multline*}
    \begin{pmatrix}
      a&x_{1}&x_{2}\tensor y_{2}\\
      \flat(x_{3})&b&y_{1}\\
      \flat(x_{4}\tensor y_{4})&\flat(y_{3})&c_{1}
    \end{pmatrix}
    \begin{pmatrix}
      x_{5}\tensor y_{5}\\ y_{6}\\c_{2}
    \end{pmatrix}
    \\=
    \begin{pmatrix}
      a\cdot x_{5}\tensor y_{5}+x_{1}\tensor y_{6}+x_{2}\tensor
      y_{2}\cdot
      c_{2} \\
      \rip B<x_{3},x_{5}>\cdot y_{5}+b\cdot y_{6}+y_{1}\cdot c_{2} \\
      \brip C <{\rip B<x_{5},x_{4}>}\cdot y_{4},y_{5}> + \rip
      C<y_{3},y_{6}>+c_{1}c_{2}
    \end{pmatrix}.
  \end{multline*}
The action of $G$ on
  $\K\bigl((X\tensor_{B}Y)\oplus Y\oplus C\bigr)$ given by
  $\operatorname{Ad}\bigl((u_{s}\tensor v_{s})\oplus
  v_{s}\oplus \gamma_{s}\bigr)$ is strongly continuous; we claim that our identification of
  $\K\bigl((X\tensor_{B}Y)\oplus Y\oplus C\bigr)$ with $F$ intertwines
  $\operatorname{Ad}\bigl((u_{s}\tensor v_{s})\oplus
  v_{s}\oplus \gamma_{s}\bigr)$ with
  \begin{equation*}
  \eta:=  \begin{pmatrix}
      \alpha&u&u\tensor v\\
      \flat (u)&\beta&v\\
      \flat(u\tensor v)& \flat (v)&\gamma
    \end{pmatrix}.
  \end{equation*}
  This is clear on the diagonal blocks, and the other entries
  follow from routine calculations.  For example, if $L_{x\tensor
    y}\in \K(C,X\tensor_{B}Y)$ maps $c$ to $x\tensor
  y\cdot c$, then
  \begin{align*}
  \operatorname{Ad}\bigl((u_{s}\tensor v_{s})\oplus
  v_{s}\oplus \gamma_{s}\bigr)(L_{x\tensor y})(c)&=
    (u_{s}\tensor v_{s})(L_{x\tensor y}( \gamma_{s}^{-1}(c))\\
    & =
    u_{s}(x)\tensor v_{s}\bigl(y\cdot \gamma_{s}^{-1}(c)\bigr)\\
    &=u_{s}(x)\tensor v_{s}(y)\cdot c\\&=L_{u_{s}(x)\tensor v_{s}(y)}(c).
  \end{align*}

  By assumption, we have nondegenerate maps $\phi_{A}:C_{0}(T)\to
  M(A)$, $\phi_{B}:C_{0}(T)\to M(B)$ and $\phi_{C}:C_{0}(T)\to M(C)$, and these give a nondegenerate diagonal map
$\phi:= \phi_{A}\oplus \phi_{B}\oplus \phi_{C}$ of $C_{0}(T)$ into $M(F)$. Now $(F,\eta,\phi)$ is an object in the
  semi-comma category
  $\mathcal{A}\bigl(G,\bigl(C_{0}(T),\operatorname{rt})\bigr)\bigr)$.
  The linking algebras $L(X)$, $L(Y)$ and
  $L(X\tensor_{B}Y)$ embed into $F$ as, respectively, the $p_{1}+p_{2}$,
  $p_{2}+p_{3}$ and $p_{1}+p_{3}$ corners.  As in
  Proposition~\ref{defgfpaib}, this allows us to transfer
  $E^{L(X)}$ to $E^{F}$ and realise $L(X)^{L(u)}$ as
  $(p_{1}+p_{2})F^{\eta}(p_{1}+p_{2})$.  Similarly, we realise
  $L(Y)^{L(v)}$ as $(p_{2}+p_{3})F^{\eta}(p_{2}+p_{3})$ and
  $L(X\tensor_{B}Y)^{L(u\tensor v)}$ as
  $(p_{1}+p_{3})F^{\eta}(p_{1}+p_{3})$.  Putting these identifications together shows that
  \begin{equation*}
    F^\eta=
    \begin{pmatrix}
      A^{\alpha}&X^{u}&(X\tensor_{B}Y)^{u\tensor v} \\ *
      &B^{\beta}&Y^{v}\\{}*&*&C^{\gamma}
    \end{pmatrix}.
  \end{equation*}
  We know that $X^{u}$, $Y^{v}$ and $(X\tensor_{B}Y)^{u\tensor v}$
  are imprimitivity bimodules with actions and inner products coming
  from the matrix operations in $F^{\eta}$, so
  Lemma~\ref{lem-lemma-e} gives the required isomorphism of
  $X^{u}\tensor_{B^{\beta}}Y^{u}$ onto $(X\tensor_{B}Y)^{u\tensor v}$.
\end{proof}

\subsection{Proof of Theorem~\ref{thm-fix}}\label{pf-fixed} To show
that $\Fix$ is a functor we need to show that it maps the identity
morphism at an object $(A,\alpha, \phi_A)$ in $\Aa(G)$ to the identity
morphism at $\Fix(A,\alpha,\phi_A)$ in $\CC$, and that $\Fix$ preserves
composition.

The identity morphism at $(A,\alpha, \phi_A)$ is implemented by the
bimodule $(_AA_A,\alpha)$, and the linking algebra $L({}_AA_A)$ is
just $M_2(A)$. Since every representation of $M_2(A)$ is equivalent to
one of the form $\pi\otimes 1$ for some representation $\pi$ of $A$,
every positive functional on $M_2(A)$ is a linear combination of
vector functionals applied to the entries of the matrices. Thus the
linear map
$\big(\begin{smallmatrix}a&b\\c&d\end{smallmatrix}\big)\mapsto
\big(\begin{smallmatrix}E^A(a)&E^A(b)\\E^A(c)&E^A(d)\end{smallmatrix}\big)$
has the property \eqref{defE} which characterises $E^{M_2(A)}$, and it
follows that
\begin{align*}
  \Fix({}_AA_A,\alpha)&=p\big(\overline{E^{M_2(A)}(C_c(T)M_2(A)C_c(T))}\big)q\\
  &=\overline{E^A(C_c(T)AC_c(T))}\\&={}_{\Fix
    A}\Fix(A,\alpha,\phi_A)_{\Fix A},
\end{align*}
which is the identity morphism at $\Fix(A,\alpha, \phi_A)$ in $\CC$.

It remains to show that if ${}_{(A,\alpha)}(X,u)_{(B,\beta)}$ and
${}_{(B,\beta)}(Y,v)_{(C,\gamma)}$ are right-Hilbert bimodules, then
$\Fix(X\otimes_BY,u\otimes v)$ is isomorphic to
$\Fix(X,u)\otimes_{B^\beta}\Fix(Y,v)$ as a right-Hilbert
$A^\alpha$\,--\,$C^\gamma$ bimodule. To form
$\Fix(X,u)\otimes_{B^\beta}\Fix(Y,v)$, we factor $(X,u)$ and $(Y,v)$
as in Corollary~\ref{factor}. Applying Proposition~2.6 of \cite{kqrproper}
gives us nondegenerate homomorphisms $\kappa_A|$ and $\kappa_B|$, and
Proposition~\ref{defgfpaib} gives imprimitivity bimodules $X^u$ and
$Y^v$, and then $\Fix(X,u)\otimes_{B^\beta}\Fix(Y,v)$ is implemented
by the composition
\begin{equation}\label{eqfunct1a}
  \xymatrix{
    A^\alpha
    \ar[r]^-{\kappa_A|}
    &\K(X)^\mu
    \ar[r]^-{X^u}
    &B^\beta
    \ar[r]^-{\kappa_B|}
    &\K(Y)^\notnu
    \ar[r]^-{Y^v}
    &C^\gamma,
  }
\end{equation}
where $\mu$ and $\notnu$ are the compatible actions coming from
Corollary~~\ref{factor}.  Similarly, $\Fix(X\otimes_BY,u\otimes v)$ is
implemented by the composition
\begin{equation}\label{eqfunct2a}
  \xymatrix{
    A^\alpha
    \ar[r]^-{\kappa_A'|}
    &\K(X\otimes_BY)^\xi
    \ar[rr]^-{(X\otimes_BY)^{u\otimes v}}
    &&C^\gamma
  }
\end{equation}
where $\xi$ is the compatible action coming from
Corollary~~\ref{factor}. We need to show that the right-Hilbert
$A^\alpha$\,--\,$C^\gamma$ bimodules defined by \eqref{eqfunct1a} and
\eqref{eqfunct2a} are isomorphic.

We begin by applying Proposition~\ref{prop-lemma-a} to the middle two
arrows of \eqref{eqfunct1a}. This gives us a right-Hilbert bimodule
isomorphism of $X^u\otimes_{B^\beta} \K(Y)^\rho$ onto
$(X\otimes_B\K(Y)\big)^{u\otimes\notnu}$, and hence the composition
\eqref{eqfunct1a} is isomorphic to
\begin{equation}\label{eqfunct1b}
  \xymatrix{
    A^\alpha
    \ar[r]^-{\kappa_A|}
    &\K(X)^\mu
    \ar[rr]^-{\kappa_{\K(X)}|}
    &&\K((X\otimes_B\K(Y))^{u\otimes\notnu}
    \ar[rr]^-{(X\otimes_B\K(Y))^{u\otimes\notnu}}
    &&\K(Y)^\notnu
    \ar[r]^-{Y^v}
    &C^\gamma.
  }
\end{equation}
Applying Theorem~\ref{thm-lemma-b} to the third and fourth arrows in
\eqref{eqfunct1b} gives an isomorphism
\begin{equation}\label{eqfunct3}
  \big(X\otimes_B\K(Y)\big)^{u\otimes\notnu}\otimes_{\K(Y)^\notnu}Y^v\cong \big((X\otimes_B\K(Y))\otimes_{\K(Y)}Y\big)^{(u\otimes\notnu)\otimes v}
\end{equation}
of imprimitivity bimodules; the composition of the first two arrows
just implements the left action by $A$ (or rather, of the subalgebra $A^\alpha$ of $M(A)$) on $M(X\otimes_B\K(Y))$. So the
compositions \eqref{eqfunct1a} and \eqref{eqfunct1b} are isomorphic to
\begin{equation*}\label{eqfunct1c}
  \xymatrix{
    A^\alpha
    \ar[r]
    &\K((X\otimes_B\K(Y)\otimes_{\K(Y)} Y)^{u\otimes\notnu\otimes v})
    \ar[rrrr]^-{(X\otimes_B\K(Y)\otimes_{\K(Y)} Y)^{u\otimes\notnu\otimes v}}
    &&&&C^\gamma
  }
\end{equation*}
as right-Hilbert bimodules, and we need a
right-Hilbert $C^\gamma$-module isomorphism of the right-hand side of
\eqref{eqfunct3} onto $(X\otimes_BY)^{u\otimes v}$ which respects the
left actions of $A^\alpha$.

The map $(x\otimes\Theta)\otimes y\mapsto x\otimes\Theta(y)$ extends
to a right $(C,\gamma)$-module isomorphism $\theta$ of
$\big(X\otimes_B\K(Y)\otimes_{\K(Y)}Y, (u\otimes\notnu)\otimes v\big)$
onto $(X\otimes_BY,u\otimes v)$, and $\theta$ preserves the left actions of
$A$. This isomorphism induces an equivariant isomorphism $\Ad\theta$
of $\K(X\otimes_B\K(Y)\otimes_{\K(Y)}Y)$ onto $\K(X\otimes_B Y)$. Proposition~\ref{defgfpaib} implies that the
two linking algebras are objects in $(\Aa(G), (C_0(T),\rt))$, and
\[
L(\theta)=\begin{pmatrix}\Ad\theta&\theta\\ *
  &\id\end{pmatrix}:\big(L(X\otimes_B\K(Y)\otimes_{\K(Y)}Y),L(u\otimes\notnu\otimes
v)\big)\to \big(L(X\otimes_B Y),u\otimes v\big)
\]
is an isomorphism. By \cite[Proposition~2.6]{kqrproper}, $L(\theta)$
induces an isomorphism $L(\theta)|$ on generalised fixed-point
algebras; by Proposition~\ref{defgfpaib}, the three corners
$(\Ad\theta)|$, $\theta|$ and  $\id$ of $L(\theta)|$ form an imprimitivity-bimodule isomorphism of the right-hand side of \eqref{eqfunct3} onto
$(X\otimes_B Y)^{u\otimes v}$. This isomorphism preserves the left
action of $A^\alpha$ because $\theta$
preserves the left action of $A$.

This completes the proof of Theorem~\ref{thm-fix}.

\subsection{Proof of Theorem~\ref{thm-main}}\label{pfthmmain} Let
$(X,u): (A,\alpha,\phi_A)\to (B,\beta,\phi_B)$ be a morphism in
$\Aa(G,(C_0(T),\rt))$. We factor $X$ using Corollary~\ref{factor} to
get:
\begin{equation}\label{eq-main1}
  \xymatrix{
    (A,\alpha,\phi_A)
    \ar[r]^-{\kappa}
    &(\K(X),\mu,\phi_{\K(X)})
    \ar[r]^-{(X,u)}
    &(B,\beta,\phi_B).
  }
\end{equation}
Write $\K=\K(X)$. We need to show that the outer square of the
following diagram
\begin{equation}\label{eq-main-2}
  \xymatrix{
    A\rtimes_{r} G
    \ar[rrrr]^-{Z(A,\alpha,\phi_A)}
    \ar[dd]_-{X\rtimes_{r}G}
    \ar[rd]^-{\kappa\rtimes_r G}
    &&&&
    \Fix(A,\alpha,\phi_A)
    \ar[ld]_-{\kappa|}
    \ar[dd]^-{\Fix(X,u)}
    \\
    & \K\rtimes_{r} G
    \ar[ld]^-{X_B\rtimes_r G}
    \ar[rr]_-{Z(\K,\mu,\phi_{\K})}
    &&\Fix(\K,\mu,\phi_\K)
    \ar[rd]^-{({}_\K X_B)^u}
    \\
    B\rtimes_{r}G
    \ar[rrrr]^-{Z(B,\beta,\phi_B)}
    &&&&
    \Fix(B,\beta,\phi_B)
  }
\end{equation}
commutes.  The factorisation \eqref{eq-main1} induces a factorisation
of $X\rtimes_{u,r}G$, and this factorisation says precisely that the
left triangle of diagram \eqref{eq-main-2} commutes.  Since
$\kappa:(A,\alpha,\phi_A)\to (\K(X),\mu,\phi_\kappa)$ is a morphism in
$(C_0(T),\rt)\downarrow \Aa_{\nd}(G)$,  the upper
quadrilateral of diagram \eqref{eq-main-2} commutes by
\cite[Theorem~3.2]{kqrproper}.

Next,  recall that $L( X)\rtimes_{L(u),r}G$ is isomorphic to $L(
X\rtimes_{u,r}G)$ (see, for example,
\cite[Proposition~4.3]{aHRWproper}) and that $L(X)^{L(u)}=L(X^u)$ by
Proposition~\ref{defgfpaib}. So we can view $Z(L(X))$ as a
$L(X\rtimes_{u,r}G)$\,--\,$L(X^u)$ imprimitivity bimodule, and apply
\cite[Lemma~4.6]{ER} to see that that the lower quadrilateral of
diagram \eqref{eq-main-2} commutes.  Since $\Fix(X,u)$ is by
definition the composition of $\kappa|$ with $X^u$, the right
triangle of diagram \eqref{eq-main-2} commutes. Thus
the outer square commutes too.


\section{Naturality of Mansfield imprimitivity}\label{sec-appl}

We want to use Theorem~\ref{thm-main} to prove naturality for the Mansfield im\-prim\-itivity theorem for closed subgroups of
\cite{hrman}. As observed in \cite{kqrproper}, for every closed
subgroup $H$, $((B\rtimes_\delta G,\hat\delta|),j_G)$ is an object in
the comma category $(C_0(G),\rt)\downarrow \Aa_{\nd}(H)$, and the general theory of \cite[\S3]{kqrproper} gives a version of naturality for functors
defined on a category of coactions built from the smaller category
$\CCnd$, but taking values in the larger category $\CC$. Here we prove a similar result for functors defined on a category of coactions built from $\CC$ (Theorem~\ref{newnatforcoact}
below). This new theorem directly extends \cite[Theorem~4.3]{enchilada} from closed normal subgroups to arbitrary closed subgroups.
Throughout this section, $H$ is a closed subgroup of a locally compact group $G$.

Mansfield's theorem is about crossed products by coactions, and since
there are several different kinds of coactions, we either have to
prove several versions of every theorem or make choices. Since the
crossed products associated to the different kinds of coactions
typically coincide, and since our goal is a theorem about crossed
products, we are going to make choices and hope that whoever wants
them can deduce the other versions easily enough. Here we choose to
use the \emph{normal} coactions which were introduced in \cite{quigg-fr} and discussed at length in \cite[Appendix~A]{enchilada}. We have three reasons for making this choice. First, we know from
\cite[Theorem~2.15]{enchilada} that there is a category $\Aca^{\nor}(G)$ (denoted there by $\Cc^n(G)$) in
which the objects $(B,\delta)$ consist of a normal coaction
$\delta$ of $G$ on $B$, and in which the morphisms are based on those
in the category $\CC$. Second, we know from
\cite[Theorem~3.13]{enchilada} that the assignment $(B,\delta)\mapsto
B\rtimes_\delta G$ extends to a functor $\CP$ from $\Aca^{\nor}(G)$ to
$\CC$. And third, we want our results to directly generalise
Theorem~4.3 of \cite{enchilada}, which is couched in terms of normal
coactions and the crossed product functors on $\Aca^{\nor}(G)$. However, our
choice has downsides: we have to jump around a bit to apply results from \cite{hrman} and \cite{kqrproper}
about reduced coactions, and we have to add the hypothesis of normality, which irks a little because we believe we could prove a similar result
without normality using Quigg's normalisation process \cite{quigg-fr}.

Suppose that $\delta:B\to M(B\otimes C^*(G))$ is a normal
coaction and that $(j_B,j_G^B)$ is the canonical covariant
representation of $(B,\delta)$ in $M(B\rtimes_\delta G)$. Then $j_G=j^B_G$
is equivariant for the action $\rt$ of $H$ by right translation on
$C_0(G)$ and the restriction $\hat\delta|_H$ of the dual coaction, and
hence $(B\rtimes_\delta G,\hat\delta|_H,j_G)$ is an object in the
semi-comma category $\Aa(H,(C_0(G),\rt))$. Since the morphisms in the
semi-comma category are just those in $\Aa(H)$, it follows from
\cite[Theorem~3.13]{enchilada} that there is a functor $\CP:\Aca^{\nor}(G)\to
\Aa(H,(C_0(G),\rt))$ which takes  $(B,\delta)$ to
$(B\rtimes_\delta G,\hat\delta|_H,j_G)$. The natural isomorphism of
Theorem~\ref{thm-main} immediately gives:

\begin{cor}\label{1stMan}
  The assignment $(B,\delta)\mapsto Z(B\rtimes_\delta
  G,\hat\delta|_H,j_G)$ implements a natural isomorphism between the
  functors $\RCP_H\circ\CP$ and $\Fix\circ\CP$ from $\Aca^{\nor}(G)$ to
  $\CC$.
\end{cor}

Suppose for the moment that $G$ is amenable, so that full and reduced
crossed products coincide, and normal and reduced coactions
coincide. Then it follows from results in \cite{hrman} that the
fixed-point algebra $\Fix(B\rtimes_\delta G,\hat\delta|_H,j_G)$ is the
crossed product $B\rtimes_\delta(G/H)$ by the homogeneous space $G/H$
(see \cite[\S6]{kqrproper} for the details), and
Corollary~\ref{1stMan} gives a natural isomorphism between functors
with object maps $(B,\delta)\mapsto
(B\rtimes_\delta G)\rtimes_{\hat\delta}H$ and
$\RCP_{G/H}:(B,\delta)\mapsto B\rtimes_\delta(G/H)$. This isomorphism
extends \cite[Theorem~4.3]{enchilada} to non-normal subgroups of
amenable groups. Our next goal is to remove the assumption of
amenability, and thereby obtain a theorem which includes the full
strength of \cite[Theorem~4.3]{enchilada}.

The amenability of $G$ appeared above because we needed to use the
results on functoriality of $\CP$ from \cite{enchilada}, which apply
to normal coactions, alongside results from \cite{hrman}, which are
about reduced coactions. So to lift the amenability hypothesis, we
have to convert the results about reduced coactions to normal
coactions.

Suppose again that $\delta:B\to M(B\otimes C^*(G))$ is a normal
coaction. Then $\delta^r:=(\id\otimes\lambda)\circ\delta:B\to
M(B\otimes C_r^*(G))$ is a reduced coaction, called the
\emph{reduction} of $\delta$, and
$(B\rtimes_{\delta^r}G,j_B,j_G)$ is also a crossed product for the
original system $(B,G,\delta)$ \cite[Propositions~3.3
and~2.8]{quigg-fr}. The dual actions $\hat{\delta}$ and
$\widehat{\delta^r}$ also coincide, because both are characterised by
their behaviour on the spanning set $\{j_B(b)j_G(f)\}$, and
\[
\hat\delta_s(j_B(b)j_G(f))=j_B(b)j_G(\rt_s(f))=(\widehat{\delta^r})_s(j_B(b)j_G(f)).
\]
Thus
$(B\rtimes_{\delta}G,\hat{\delta}|_H,j_G)=(B\rtimes_{\delta^r}G,\widehat{\delta^r}|_H,j_G)$. This
system is an element of the semi-comma category $\Aa(H,(C_0(G),\rt))$,
which has the same objects as the comma category
$(C_0(G),\rt|)\downarrow \Aa_{\nd}(H)$ in \cite{kqrproper}, so
Proposition~3.1 of \cite{kqrproper} implies that $\hat{\delta}|_H$ is
proper with $\Fix(B\rtimes_\delta
G,\hat\delta|,j_G)=\Fix(B\rtimes_{\delta^r}
G,\widehat{\delta^r}|,j_G)$ (properness itself was first proved in
\cite[Theorem~5.7]{integrable}). Since $H$ acts freely on $G$, the
action $\hat\delta|_H$ is also saturated in the sense of
\cite[Definition~1.6]{proper}, and Corollary~1.7 of \cite{proper}
gives an imprimitivity bimodule
\[
Z(B\rtimes_{\delta}G,\hat{\delta}|_H,j_G)=Z(B\rtimes_{\delta^r}G,\widehat{\delta^r}|_H,j_G)
\]
implementing  a Morita equivalence between $(B\rtimes_\delta
G)\rtimes_{\hat\delta,r}H$ and $\Fix(B\rtimes_\delta
G,\hat\delta|_H,j_G)$. The discussion at the start of
\cite[\S6]{kqrproper} (applied to the reduced coaction $\delta^r$)
gives the following description of $\Fix(B\rtimes_\delta
G,\hat\delta|_H,j_G)$.

\begin{prop}
  As a subset of of $M(B\rtimes_{\delta}G)=M(B\rtimes_{\delta^r} G)$,
  the generalised fixed-point algebra $\Fix(B\rtimes_\delta
  G,\hat\delta|_H,j_G):=\Fix(B\rtimes_{\delta^r}
  G,\widehat{\delta^r}|_H,j_G)$ coincides with the subset
  \begin{equation}\label{spanforBG/H}
    B\rtimes_{\delta,r}(G/H):=\clsp\{j_B(b)j_G(f):b\in B,\ f\in C_0(G/H)\},
  \end{equation}
  which is called the \emph{crossed product of $B$ by the homogeneous space $G/H$}.
\end{prop}

We have now proved an analogue of \cite[Theorem~3.1]{hrman}
for normal coactions:

\begin{prop}[Mansfield imprimitivity for normal coactions]
  \label{ahrmansfieldnormal}Suppose that $\delta$ is a normal coaction
  of $G$ on a $C^*$-algebra $B$, and that $H$ is a closed subgroup of
  $G$. Then the dual action $\hat\delta$ of $H$ on $B\rtimes_\delta G$
  is proper and saturated with respect to the subalgebra
  $A_0=j_G(C_c(G))(B\rtimes_\delta G)j_G(C_c(G))$, and Rieffel's
  imprimitivity bimodule
  $Z(B\rtimes_{\delta}G,\hat{\delta}|_H,j_G)=\overline{A_0}$
  implements a Morita equivalence between $(B\rtimes_{\delta}
  G)\rtimes_{\hat\delta,r}H$ and $B\rtimes_{\delta,r}(G/H)$.
\end{prop}

\begin{remark}\label{ours=man}
  The algebras $(B\rtimes_{\delta}
  G)\rtimes_{\hat\delta,r}H$ and $B\rtimes_{\delta,r}(G/H)$ in
  Proposition~\ref{ahrmansfieldnormal} are the same as those in
  \cite[Theorem~3.1]{hrman} (for the reduction $\delta^r$). The discussion in \cite[\S 6]{kqrproper}
  shows that the bimodule $Z(B\rtimes_{\delta}G,\hat{\delta}|_H,j_G)$
  is the same as the bimodule $\overline{\DD}$ constructed in
  \cite[Theorem~3.1]{hrman}. Thus when $H$ is normal,
  $Z(B\rtimes_{\delta}G,\hat{\delta}|_H,j_G)$ is the usual Mansfield
  bimodule $Y_{G/H}^G$ appearing in \cite{man}
  (if $H$ is amenable) or in \cite[\S3]{kqman}.
\end{remark}

To get the naturality of this version of Mansfield imprimitivity, we
need to extend the construction of crossed products by homogeneous
spaces to a functor on $\Aca^{\nor}(G)$. Since the objects
$B\rtimes_{\delta,r}(G/H)$ are the same as $\Fix\circ \CP(B,\delta)$,
applying $\Fix\circ \CP$ to the morphisms in $\Aca^{\nor}(G)$ gives such a
functor. But to see that this functor is the same as the one used in
\cite{enchilada} when $H$ is normal, we need a more concrete
description of what it does to morphisms.

\begin{prop}\label{idcpGH}
  Suppose that ${}_{(B,\delta)}(X,\Delta)_{(C,\epsilon)}$ is a
  right-Hilbert bimodule which implements a morphism in $\Aca^{\nor}(G)$,
  and define
  \begin{equation}\label{fix=xG/H}
    X\rtimes_{\Delta,r}(G/H):=\clsp\{j_X(x)j_G^C(f):x\in X,\ f\in C_0(G/H)\}\subset M(X\rtimes_\Delta G).
  \end{equation}
  Then with the module actions and inner products from
  ${}_{M(B\rtimes_\delta G)}M(X\rtimes_\Delta G)_{M(C\rtimes_\epsilon
    G)}$, the vector space $X\rtimes_{\Delta,r}(G/H)$ becomes a
  right-Hilbert $(B\rtimes_\delta G)$\,--\,$(C\rtimes_\epsilon G)$
  bimodule. The assignments
  \[(B,\delta)\mapsto B\rtimes_{\delta,r}(G/H)\ \text{ and }\
  [X,\Delta]\mapsto [X\rtimes_{\Delta,r}(G/H)]
  \]
  form a functor $\RCP_{G/H}$ from $\Aca^{\nor}(G)$ to $\CC$ which coincides
  with $\Fix\circ\CP$.
\end{prop}

\begin{proof}
  We know from \cite[Theorem~3.13]{enchilada} and
  Theorem~\ref{thm-fix} that $\Fix\circ\CP$ is a functor. So, since
  $\Fix(B\rtimes_\delta G,\hat\delta|,j_G)$ and
  $B\rtimes_{\delta,r}(G/H)$ are the same subset of $M(B\rtimes_\delta
  G)$, and similarly for $(C,\epsilon)$, it suffices to prove that
  $\Fix(X\rtimes_\Delta G,\hat\Delta|_H)\subset M(X\rtimes_\Delta G)$
  coincides with $X\rtimes_{\Delta,r}(G/H)$. For this calculation we
  adopt the notation of \cite[page~65]{enchilada}.

  The bimodule $\Fix(X\rtimes_\Delta G,\hat\Delta|_H)$ has as its
  underlying set the space
  \[
  (X\rtimes_{\Delta} G)^{\hat\Delta|_H}=p\big(\Fix(L(X\rtimes_\Delta
  G),L(\hat\Delta|_H), j_G^L)\big)q,
  \]
  where
  $j_G^L:j_G^{\K}\oplus j_G^C$. The
  canonical identification of $L(X\rtimes_\Delta G)$ with
  $L(X)\rtimes_{L(\Delta)} G$ (see \cite[Lemma~3.10]{enchilada})
  intertwines $L(\hat\Delta|_H)$ and $L(\Delta)^\wedge|_H$. The
  formula in \eqref{spanforBG/H} shows that
  \[
  \big(L(X)\rtimes_{L(\Delta)} G\big)^{L(\Delta)^\wedge|_H}
  =\clsp\{j_L(d)j_G(f):d\in L(X),\ f\in C_0(G/H)\}.
  \]
  When we pull $j_L\big(\begin{smallmatrix}k&x\\
    *&c\end{smallmatrix}\big)$ back to $L(X\rtimes_\delta G)$, we get
  the matrix $\big(\begin{smallmatrix}j_{\K}(k)&j_X(x)\\
    *&j_C(c)\end{smallmatrix}\big)$, and $j_G$ pulls back to
  $j_G^L$. So $L(X\rtimes_\Delta G)^{L(\hat\Delta|_H)}$ is
  \[
  \clsp\Big\{\begin{pmatrix}
    j_{\K}(k)j_G^{\K}(f)&j_X(x)j_G^C(f)\\
    *&j_C(c)j_G^C(f)
  \end{pmatrix}:
  \begin{pmatrix}k&x\\ *&c\end{pmatrix}\in L(X),\ f\in C_0(G/H)\Big\}.
  \]
  Thus
  \[
  (X\rtimes_{\Delta} G)^{\hat\Delta|_H} =\clsp\Big\{p\begin{pmatrix}
    j_{\K}(k)j_G^{\K}(f)&j_X(x)j_G^C(f)\\
    *&j_C(c)j_G^C(f)
  \end{pmatrix}q\Big\},
  \]
  which is the right-hand side of \eqref{fix=xG/H}.
\end{proof}

Now Corollary~\ref{1stMan} shows that Mansfield imprimitivity for
normal coactions (as in Proposition~~\ref{ahrmansfieldnormal}) is
natural. To sum up:

\begin{thm}\label{newnatforcoact}
  Suppose that $H$ is a closed subgroup of a locally compact group
  $G$. Then Rieffel's bimodules
  \[
  \{Z(B\rtimes_\delta G,\hat\delta|_H,j_G):(B,\delta)\in \Aca^{\nor}(G)\}
  \]
  implement a natural isomorphism between the functors $\RCP_H\circ
  \CP$ and $\RCP_{G/H}$ from $\Aca^{\nor}(G)$ to $\CC$.
\end{thm}

\begin{remark}
  Suppose that $H$ is normal. Then it follows from
  Proposition~\ref{idcpGH} that the bimodule
  $X\rtimes_{\Delta,r}(G/H)$ is the same as the one used in the
  construction of the crossed-product functor in
  \cite[Theorem~3.13]{enchilada} (see
  \cite[Proposition~3.9]{enchilada}). We saw in Remark~\ref{ours=man}
  that the imprimitivity bimodules $Z(B\rtimes_\delta
  G,\hat\delta|_H,j_G)$ coincide with the Mansfield bimodules
  $Y_{G/H}^G(B,\delta)$ which implement the natural isomorphism. Thus
  Theorem~\ref{newnatforcoact} is a direct generalisation of
  \cite[Theorem~4.3]{enchilada} to non-normal subgroups.
\end{remark}

\section{Proper actions on graph algebras}

Let $E=(E^0,E^1,r,s)$ be a directed graph. We assume as in \cite[\S6]{kps} that $E$ is row-finite and has no sources, and in general we use the conventions of \cite{cbms}. In particular, the graph algebra $C^*(E)$ is generated by a universal Cuntz-Krieger $E$-family $\{s_e,\;p_v:e\in E^1,\;v\in E^0\}$, and
\begin{equation}\label{defX0}
X_0(E):=\newspan\{s_\mu s_\nu^*:\mu,\nu\in E^*\}
\end{equation}
is a dense $*$-subalgebra of $C^*(E)$.

An action of a group $G$ on $E$ consists of actions on $E^0$ and $E^1$ which preserve the range and source maps $r$ and $s$, and the action is \emph{free} if the action on $E^0$ is free. An action of $G$ on $E$ induces an action $\alpha=\alpha^E$ of $G$ on $C^*(E)$ such that $\alpha_t(s_e)=s_{e\cdot t^{-1}}$. It was shown in \cite[\S1]{pr} that if $G$ acts freely on $E$, then the induced action $\alpha$ is proper and saturated with respect to the subalgebra $X_0(E)$, and the generalised fixed-point algebra is isomorphic to the $C^*$-algebra of the quotient graph $E/G$ \cite[Corollary~1.5]{pr}. Thus Rieffel's theory gives an imprimitivity bimodule $Z(E,G)$ which implements a Morita equivalence between $C^*(E)\rtimes_{\alpha,r}G$ and $C^*(E/G)$ \cite[Theorem~1.6]{pr} (thus giving a new proof of a theorem of Kumjian and Pask \cite{kpaction}).

Here we formulate a naturality result for the Morita equivalence of \cite{pr}, using a category of directed graphs recently introduced by Kumjian, Pask and Sims \cite{kps}. If $E$ and $F$ are directed graphs, an \emph{$E$--$F$ morph} is a countable set $X$ together with maps $r=r^X:X\to E^0$ and $s=s^X:X\to F^0$, and a bijection $b=b^X$ of the fibre product $X*F^1:=\{(x,f):s^X(x)=r^F(f)\}$ onto $E^1*X=\{(e,y):s^E(e)=r^X(y)\}$ such that
\[
b(x,f)=(e,y)\Longrightarrow s^X(y)=s^F(f)\text{ and } r^E(e)=r^X(x).
\]
Proposition~6.1 of \cite{kps} says that there is a category $\Graph$ in which the objects are row-finite directed graphs with no sources, and the morphisms from $E$ to $F$ are the isomorphism classes $[X]$ of $E$--$F$ morphs such that $s^X$ and $r^X$ are surjective and $r^X$ is finite-to-one. It is straightforward to add actions of $G$ on morphs (just demand that they preserve all the structure) to get a category $\AGraph(G)$ in which the objects $(E,G)$ consist of a directed graph $E$ with a right action of $G$. We are also interested in the full subcategory $\FAGraph(G)$ in which the actions of $G$ on objects are free.

Kumjian, Pask and Sims  showed that for each $E$--$F$ morph $X$ there is a right-Hilbert $C^*(E)$\,--\,$C^*(F)$ bimodule $C^*(X)$, whose isomorphism class depends only on the isomorphism class of $X$, and that the assignments $E\mapsto C^*(E)$, $[X]\mapsto [C^*(X)]$ form a functor\footnote{In \cite{kps}, the category $\Graph$ is denoted $\mathcal{M}_1^{\maltese}$, and the functor $\CK$ is denoted $\mathcal{H}$. We have chosen our name because $\CK$ associates to each graph its Cuntz-Krieger algebra $C^*(E)$.} $\CK$ from $\Graph$ to the category $\CC$ (see \cite[Theorem~6.6]{kps}; we discuss their construction  in the proof of Proposition~\ref{fix=quot}). If $(E,G)$ and $(F,G)$ are objects in $\AGraph(G)$, then an action of $G$ on an $E$--$F$ morph $X$ induces an $\alpha^E$--$\alpha^F$ compatible action $\alpha^X$ of $G$ on $C^*(X)$, and hence we can extend $\CK$ to a functor $\CK(G)$ from $\AGraph(G)$ to $\Aa(G)$. On the other hand, passing to the quotient graph also gives a functor $\Qq^G$ from $\AGraph(G)$ to $\Graph$. We can now formulate our theorem:

\begin{thm}\label{prnatural}
The assignment $(E,G)\mapsto Z(E,G)$ implements a natural isomorphism between the functors $\RCP\circ \CK(G)$ and $\CK\circ \Qq^G$ from $\FAGraph(G)$ to $\CC$.
\end{thm}

To apply Theorem~\ref{thm-main} we need an appropriate semi-comma category. Fix a countable set $S$ with a free right action of $G$. Then we define the \emph{semi-comma category} $\AGraph(G,S)$ to be the category whose objects are triples $(E,G,\pi)$, where $(E,G)$ is an object of $\AGraph(G)$ and $\pi\colon E^0\to S$ is an equivariant surjection, and whose morphisms from $(E,G,\pi)$ to $(F,G,\rho)$ are  the morphisms from $(E,G)$ to $(F,G)$ in $\AGraph(G)$.  Then $\pi$ induces an equivariant nondegenerate homomorphism of $c_0(S)$ into $M(c_0(E^0))$, and hence an equivariant nondegenerate homomorphism $\pi^*\colon c_0(S)\to M(C^*(E))$.  Thus we obtain a functor $\CK(G)$ from $\AGraph(G,S)$ into the semi-comma category $\Aa(G,(c_0(S),\rt))$ which maps objects $(E,G,\pi)$ to $(C^*(E),\alpha^E,\pi^*)$. Applying Theorem~\ref{thm-main} gives:

\begin{prop}\label{applymainthm}
Rieffel's bimodules implement a natural isomorphism
\[
(E,G,\pi)\mapsto Z(C^*(E),\alpha,\pi^*)
\]
between the functors $\RCP\circ\CK(G)$ and $\Fix\circ\CK(G)$ from $\AGraph(G,S)$ to $\CC$.
\end{prop}

To deduce Theorem~\ref{prnatural} from Proposition~\ref{applymainthm}, we have to relate Rieffel's bimodules and the functor $\Fix$ to the bimodules and $C^*$-algebras appearing in \cite{pr}. Since $G$ acts freely on $S$, the existence of $\pi$ implies that $G$ acts freely on $E$, so we know from \cite{pr} that $G$ acts properly on $C^*(E)$ with respect to the subalgebra $X_0(E)$ of \eqref{defX0}.

\begin{lemma}\label{idibs}
Suppose $(E,G,\pi)$ is an object in $\AGraph(G,S)$. The subalgebra $X_0(E)$ is contained in $A_0:=c_c(S)C^*(E)c_c(S)$, and the generalised fixed-point algebra $C^*(E)^\alpha$ constructed in \cite[\S1]{pr} and $\Fix(C^*(E),\alpha,\pi^*)$ are the same subalgebra of $M(C^*(E))$.  The completion $Z(C^*(E),\alpha,\pi^*)$ is also a completion of $X_0(E)$, and when we  identify the two completions  they are  equal as $(C^*(E)\rtimes_{\alpha,r}G)$\,--\,$C^*(E)^\alpha$ imprimitivity bimodules.
\end{lemma}

\begin{proof}
The algebra $c_c(S)$ is spanned by the point masses $\delta_s$, and $\pi^*(\delta_s)$ is the strictly convergent sum $\sum_{\pi(v)=s} p_v$. Since every
\[
s_\mu s_\nu^*=p_{r(\mu)}s_\mu s_\nu^*p_{r(\nu)}=\pi^*(\delta_{\pi(r(\mu))})s_\mu s_\nu^*\pi^*(\delta_{\pi(r(\nu))})
\]
belongs to $A_0$, we have $X_0(E)\subset A_0$. For every $fag\in A_0$ there is a finite set $K\subset S$ such that $fag=\pi^*(\chi_K)fag\pi^*(\chi_K)$, and the continuity of $b\mapsto E^A(\pi^*(\chi_K)b\pi^*(\chi_K))$ (from \cite[Corollary~3.6]{quigg}) shows that $E^A(A_0)$ is a subset of the closure $\overline{E^A(X_0(E))}$. Since $E^A|_{X_0(E)}$ is the same as the averaging process $I_G$ used in \cite{pr} (they both multiply elements of $X_0(E)$ in the same way), $\overline{E^A(X_0(E))}$ is the generalised fixed-point algebra $C^*(E)^\alpha$ in \cite{pr}, and we can deduce that $\Fix(C^*(E),\alpha,\pi^*)$, which is by definition the closure of $E^A(A_0)$ in $M(C^*(E))$, coincides with $C^*(E)^\alpha$.

Since the averaging processes used to define them are the same, the $C^*(E)^\alpha$- and $(\Fix C^*(E))$-valued inner products coincide on $X_0(E)$, and both the completions $\overline{X_0(E)}$ of $X_0(E)$ and $Z(C^*(E),\alpha,\pi^*)$ of $A_0$ in the norms defined by this inner product are right Hilbert $C^*(E)^\alpha$-modules. The left action of $b\in c_c(G,C^*(E))\subset C^*(E)\rtimes_{\alpha,r}G$ on $x\in X_0(E)$ is given in both $X_0(E)$ and $A_0$ by $b\cdot x=\sum_{t\in G} b(t)\alpha_t(x)$, so  $\overline{X_0(E)}$ and $Z(C^*(E),\alpha,\pi^*)$ are $(C^*(E)\rtimes_{\alpha,r}G)$\,--\,$C^*(E)^\alpha$ imprimitivity bimodules for the same actions, and we can view the completion of $X_0(E)$  as the closure $\overline{X_0(E)}$ of $X_0(E)$ in $Z(C^*(E),\alpha,\pi^*)$. Since the right-hand inner product on the submodule generates $C^*(E)^\alpha$, it follows from the Rieffel correspondence that $\overline{X_0(E)}=Z(C^*(E),\alpha,\pi^*)$.
\end{proof}

Our next step is to show that we can replace $\Fix\circ\CK(G)$ in Proposition~\ref{applymainthm} by $\CK\circ \Qq^G$. We know from \cite[Proposition~1.4]{pr} that for each system $(E,G)$ in $\FAGraph(G)$, the elements $P_{v\cdot G}:=I_G(p_v)$ and $S_{e\cdot G}:=I_G(s_e)$ form a Cuntz-Krieger $(E/G)$-family in $C^*(E)^\alpha$, and the resulting homomorphism $\phi_G^E:C^*(E/G)\to C^*(E)^\alpha$ is an isomorphism \cite[Corollary~1.5]{pr}. When $(E,G,\pi)$ is an object in the semi-comma category, we can use Lemma~\ref{idibs} to view $\phi_G^E$ as an isomorphism onto $\Fix(C^*(E),\alpha,\pi^*)$; when we use this isomorphism to view $Z(C^*(E),\alpha,\pi^*)$ as a $(C^*(E)\rtimes_{\alpha,r}G)$\,--\,$C^*(E/G)$ imprimitivity bimodule, we recover the bimodule $Z(E,G)$ which implements the equivalence of \cite[Theorem~1.6]{pr} and appears in Theorem~\ref{prnatural}.

\begin{prop}\label{fix=quot}
Suppose that $G$ acts freely on $S$. Then the assignments $(E,G,\pi)\mapsto \phi^E_G$ form a natural isomorphism between the functors $\Fix\circ\CK(G)$ and $\CK\circ \Qq^G$ from $\AGraph(G,S)$ to $\CC$.
\end{prop}

\begin{proof}
We have to show that if $(X,G)$ is an $E$--$F$ morph implementing a morphism in $\AGraph(G)$ from $(E,G,\pi)$ to $(F,G,\rho)$, then the diagram
\begin{equation}\label{fixQnat}
\xymatrix{
C^*(E/G)\ar[rrr]^{C^*(X/G)}\ar[d]_{\phi_G^E}
&&&C^*(F/G)\ar[d]^{\phi_G^F}\\
\Fix(C^*(E),\alpha^E,\pi^*)\ar[rrr]^{\Fix(C^*(X),\alpha^X)}&&&\Fix(C^*(F),\alpha^F,\rho^*)
}
\end{equation}
commutes in $\CC$.

At this point, we review the construction of the bimodules $C^*(X)$ and the action $\alpha^X$. The \emph{linking graph} $\Lambda=\Lambda(X)$ of $X$ is a $2$-graph with vertex set $\Lambda^0:=E^0\sqcup F^0$, with blue edges $\Lambda^{e_1}:=E^1\sqcup F^1$ (with the usual ranges and sources), with red edges $\Lambda^{e_2}:=X$ (with ranges and sources determined by $r^X$ and $s^X$), and with the factorisation property determined by $b^X$ via $xf=ey\Longleftrightarrow b(x,f)=(e,y)$; because $b^X$ is a bijection from the set $X*F^1$ of red-blue paths onto the set $E^1*X$ of blue-red paths, we know from \cite[\S6]{kp} that there is exactly one $2$-graph with these commuting squares. This $2$-graph is row-finite and locally convex in the sense of \cite{rsy}, and its  $C^*$-algebra $C^*(\Lambda(X))$, defined in \cite{rsy}, is generated by a universal Cuntz-Krieger family $\{s_\lambda:\lambda\in \Lambda^*\}$ which satisfies a blue Cuntz-Krieger relation at vertices in $F^0$ but no red one. As usual, $C^*(\Lambda(X))$ is spanned by $\{s_\lambda s_\mu^*:\lambda,\mu\in \Lambda^*\}$.

It is shown in \cite[Lemma~6.2]{kps} that $p_E:=\sum_{v\in E^0}s_v$ and $p_F:=\sum_{v\in F^0}s_v$ converge strictly in $M(C^*(\Lambda))$ to complementary full projections, and that there are nondegenerate injections $\iota_F$ of $C^*(F)$ onto $p_FC^*(\Lambda)p_F$ and $\iota_E$ of $C^*(E)$ into (but not necessarily onto) $p_EC^*(\Lambda)p_E$. We use $\iota_F$ to identify $C^*(F)$ with the corner in $C^*(\Lambda)$. Then the corner $p_EC^*(\Lambda)p_F$ is a full right-Hilbert $C^*(E)$\,--\,$C^*(F)$ bimodule, which we denote by $C^*(X)$. Theorem~6.6 of \cite{kps} says that the assignment $[X]\mapsto [C^*(X)]$ makes $\CK$ into a contravariant functor from $\Graph$ to $\CC$. The free action of $G$ on $X$ induces a free action of $G$ on the linking graph $\Lambda$, and this in turn induces an action $\alpha^\Lambda$ on $C^*(\Lambda)$ which is characterised by $\alpha^\Lambda_t(s_\lambda)=s_{\lambda\cdot t^{-1}}$. The projections $p_E$ and $p_F$ are fixed by $\alpha^\Lambda$, and hence $\alpha^\Lambda$ restricts to an action $\alpha^X$ on the corner $C^*(X)$.

The linking algebra of the right Hilbert $C^*(F)$-module $C^*(X)$ is by definition $C^*(\Lambda)$, and the action $L(\alpha^X)$ is $\alpha^\Lambda$. With the diagonal embedding $\phi_\Lambda:=(\iota_E\circ\pi^*)\oplus\rho^*$ of $c_0(S)$, $(C^*(\Lambda),\alpha^\Lambda, \phi_\Lambda)$ is an element of the semi-comma category; the morphism $\Fix(C^*(X),\alpha^X)$ in the bottom row of \eqref{fixQnat} has underlying right Hilbert module $p_E\Fix(C^*(\Lambda),\alpha^\Lambda, \phi_\Lambda)p_F$, and the left action of $\Fix C^*(E)$ is given by the restriction $\iota_E|$ of the nondegenerate homomorphism $\iota_E$.

The module $C^*(X/G)$ in the top row of \eqref{fixQnat} is obtained by applying the construction described above to the linking graph $\Lambda(X/G)$ of the quotient morph, which we can identify with the quotient $2$-graph $\Lambda(X)/G$. As in \cite[Lemma~1.4]{pr}, the averages $\{E^A(s_\lambda):\lambda\cdot G\in\Lambda(X)/G\}$ form a Cuntz-Krieger $(\Lambda(X)/G)$-family in $M(C^*(\Lambda(X))$, and, as in \cite[Corollary~1.5]{pr}, the induced map $\phi_G^\Lambda:C^*(\Lambda(X/G))\to M(C^*(\Lambda(X)))$ is an isomorphism onto $\Fix C^*(\Lambda(X))$.  This isomorphism maps $p_{E/G}$ and $p_{F/G}$ to $p_E$ and $p_F$, and hence it maps $C^*(X/G)$, which is by definition the corner $p_{E/G}C^*(\Lambda(X/G))p_{F/G}$, onto $\Fix(C^*(X),\alpha^X):=p_E(\Fix C^*(\Lambda(X)))p_F$. The restriction of $\phi_G^\Lambda$ to $C^*(F/G)=p_{F/G}C^*(\Lambda(X/G))p_{F/G}$ is the isomorphism $\phi_G^F$ of $C^*(F/G)$ onto $\Fix C^*(F)=p_F(\Fix C^*(\Lambda(X)))p_F$. The left action of $C^*(E/G)$ on $C^*(X/G)$ is defined via the embedding $\iota_{E/G}$ of $C^*(E/G)$ in $p_{E/G}C^*(\Lambda(X)/G)p_{E/G}$, and we can check on generators that $\iota_E|\circ\phi_G^E=\phi_G^\Lambda\circ\iota_{E/G}$. Thus the isomorphism $\phi_G^\Lambda$ is $\phi_G^E$--$\phi_G^F$ compatible, and hence is an isomorphism of right-Hilbert bimodules. But the existence of such an isomorphism implies that the diagram \eqref{fixQnat} commutes in the category $\CC$.
\end{proof}

\begin{proof}[Proof of Theorem~\ref{prnatural}]
Suppose $G$ acts freely on $E$ and $F$ and $(X,G)$ is an $(E,G)$--$(F,G)$ morph. We need to show that the following diagram commutes in $\CC$:
\begin{equation}\label{prnatdiag}
\xymatrix{
C^*(E)\rtimes_{\alpha^E,r}G\ar[rrr]^{C^*(X)\rtimes_{\alpha^X,r}G}\ar[d]_{Z(E,G)}
&&&C^*(F)\rtimes_{\alpha^F,r}G\ar[d]^{Z(F,G)}\\
C^*(E/G)\ar[rrr]^{C^*(X/G)}&&&C^*(F/G).
}
\end{equation}

Since $G$ acts freely on the discrete sets $E^0$ and $F^0$, they are trivial as $G$-bundles, and we can find equivariant surjections $\pi:E^0\to G$ and $\rho:F^0\to G$. Then with $S:=G$, $[X,G]:(E,G,\pi)\to (F,G,\rho)$ is a morphism in the semi-comma category $\AGraph(G,S)$, and Proposition~\ref{applymainthm} implies that the top square in the following diagram commutes in $\CC$:
\begin{equation}\label{combine}
\xymatrix{
C^*(E)\rtimes_{\alpha^E,r}G\ar[rrr]^{C^*(X)\rtimes_{\alpha^X,r}G}\ar[d]_{Z(C^*(E),\alpha^E,\pi^*)}
&&&C^*(F)\rtimes_{\alpha^F,r}G\ar[d]^{Z(C^*(F),\alpha^F,\rho^*)}\\
\Fix(C^*(E),\alpha^E,\pi^*)\ar[rrr]^{\Fix(C^*(X),\alpha^X)}\ar[d]_{(\phi_G^E)^{-1}}&&&\Fix(C^*(F),\alpha^F,\rho^*)\ar[d]^{(\phi_G^F)^{-1}}\\
C^*(E/G)\ar[rrr]^{C^*(X/G)}
&&&C^*(F/G)
}
\end{equation}
Proposition~\ref{fix=quot} says that the bottom square commutes in $\CC$, and hence so does the outside square. In the discussion preceding Proposition~\ref{fix=quot} we observed that using the isomorphism $\phi_G^E$ to view $Z(C^*(E),\alpha^E,\pi^*)$ as a $(C^*(E)\rtimes_{\alpha^E,r}G)$\,--\,$C^*(E/G)$ bimodule gives $Z(E,G)$, and this says precisely that the composition of the two vertical arrows on the left is $Z(E,G)$. Similarly, the composition on the right is $Z(F,G)$. So the outside square in \eqref{combine} is \eqref{prnatdiag}, and we are done.
\end{proof}

\begin{remark}
It is a little unnerving that we have \emph{chosen} the equivariant surjections $\pi$ and $\rho$ quite arbitrarily in the above proof. However, Lemma~\ref{idibs} provides some comfort: different choices give the same fixed-point algebra and the same imprimitivity bimodule, and so do not change the vertical arrows in the top square of \eqref{combine}.
\end{remark}

\end{document}